\documentclass[letterpaper,12pt,twoside]{article}
\usepackage{amsfonts}
\usepackage{amsmath}
\usepackage{amsopn}
\usepackage{amsthm}
\usepackage{amssymb}
\usepackage{graphicx}
\usepackage{rotating}
\usepackage[]{natbib}
\usepackage[english]{babel}
\usepackage[utf8]{inputenc}
\usepackage[colorlinks=true,citecolor=blue,linkcolor=blue]{hyperref}
\usepackage{xspace}
\usepackage{setspace}
\usepackage{dsfont}
\usepackage{booktabs}
\usepackage[T1]{fontenc}
\usepackage{anysize}
\usepackage{enumerate}
\usepackage[labelfont=bf, labelsep=period]{caption}
\marginsize{1.87cm}{1.87cm}{1.87cm}{1.87cm} 
\usepackage{fancyhdr}
\RequirePackage{txfonts}
\usepackage{times}
\usepackage{bm}
\usepackage[linkcolor=blue,colorlinks=true]{hyperref}

\newtheorem{Proposition}{Proposition}

\newtheorem{Theorem}{Theorem}
\newtheorem{Lemma}{Lemma}
\newtheorem{Assumption}{Assumption}
\DeclareMathOperator{\re}{\mathbb{R}}
\DeclareMathOperator{\na}{\mathbb{N}}
\newcommand{\E}{\mathbb{E}}
\newcommand{\ind}{\mathds{1}}
\newcommand{\e}{\mathrm{e}}
\newcommand{\Prob}{\mathbb{P}}
\def\bx{\mathbf{x}}
\def\bX{\mathbf{X}}
\def\by{\mathbf{y}}
\def\bY{\mathbf{Y}}
\def\d{\mathrm{d}}
\def\bbeta{\boldsymbol\beta}

\def\Oset{\mathrm{O}}
\newcommand{\I}{\mathbf{I}}
\def\y0{\mathrm{y}_0}
\newcommand{\zl}{z_{\text{l}}}
\newcommand{\zr}{z_{\text{r}}}
\newcommand{\lambdar}{\lambda_{\text{r}}}
\newcommand{\lambdal}{\lambda_{\text{l}}}

\allowdisplaybreaks

\begin{document}

\def\figureautorefname{Figure}
\def\algorithmautorefname{Algorithm}
\def\sectionautorefname{Section}
\def\subsectionautorefname{Section}
\def\subsubsectionautorefname{Section}
\def\Propositionautorefname{Proposition}
\def\Theoremautorefname{Theorem}
\def\Lemmaautorefname{Lemma}
\def\Corollaryautorefname{Corollary}
\def\Exampleautorefname{Example}
\def\Remarkautorefname{Remark}
\def\Assumptionautorefname{Assumption}
\renewcommand*\footnoterule{}

\title{Simple proof of robustness for Bayesian heavy-tailed linear regression models}

\author{Philippe Gagnon$^{1}$}

\maketitle

\thispagestyle{empty}

\noindent $^{1}$Department of Mathematics and Statistics, Universit\'{e} de Montr\'{e}al.

\begin{abstract}
 In the Bayesian literature, a line of research called \textit{resolution of conflict} is about the characterization of robustness against outliers of statistical models. The robustness characterization of a model is achieved by establishing the limiting behaviour of the posterior distribution under an asymptotic framework in which the outliers move away from the bulk of the data. The proofs of the robustness characterization results, especially the recent ones for regression models, are technical and not intuitive, limiting the accessibility and preventing the development of theory in that line of research. In this paper, we highlight that the proof complexity is due to the generality of the assumptions on the prior distribution. To address the issue of accessibility, we present a significantly simpler proof for a linear regression model with a specific class of prior distributions, among which we find typically used prior distributions. The class of prior distributions is such that each regression coefficient has a sub-exponential distribution, which allows to exploit a tail bound, contrarily to previous approaches. The proof is intuitive and uses classical results of probability theory. The generality of the assumption on the error distribution is also appealing; essentially, it can be any distribution with regularly varying or log-regularly varying tails. So far, there does not exist a result in such generality for models with regularly varying distributions. We also investigate the necessity of the assumptions. To promote the development of theory in resolution of conflict, we highlight how the key steps of the proof can be adapted for other models and present an application of the proof technique in the context of generalized linear models.
\end{abstract}


\section{Introduction}\label{sec:intro}

The topic of robustness against outliers is classical in statistics. An objective when studying this topic is to evaluate whether commonly used statistical methods are robust against outliers or not. A method is deemed not robust if a single observation can have an arbitrary impact on the estimation. A canonical example of a non-robust method is a linear regression with normal errors, as seen in \autoref{fig:mean}. In this figure, we present the results of a simple numerical experiment\footnote{The code to reproduce our numerical results is available online (see ancillary files on \url{https://arxiv.org/abs/2501.06349}).} based on $n = 20$ observations $y_1, \ldots, y_n \in \re$ of a dependent variable and $n$ data points $(x_1, x_2, \ldots, x_n) = (1, 2, \ldots, n)$ of an explanatory variable. The observations $y_1, \ldots, y_n$ were first sampled using a linear regression model with an intercept and slope coefficients both equal to 1 and independent errors each having a standard normal distribution. The observation $y_n$ was then gradually increased to obtain a sequence of data sets. For each data set, the slope coefficient is estimated using the posterior mean in a Bayesian analysis; see \autoref{sec:numerical} for the details. In \autoref{fig:mean}, we also present estimation results for the Bayesian Student's $t$ linear regression, which is the preferred Bayesian robust alternative. We observe in \autoref{fig:mean} that the estimated regression line associated with the normal model is attracted by the outlier artificially moving towards infinity, while, in contrast, that associated with the Student's $t$ model is not. This allows to conclude that the former method is non-robust while the latter is. This important distinction between these two methods is a consequence of a different tail decay: the exponential decay of the normal probability density function (PDF) leads to extremely small probabilities of observing such an extreme data point, making the normal unadapted to its presence.

\begin{figure}[ht]
 \centering
   \includegraphics[width=0.40\textwidth]{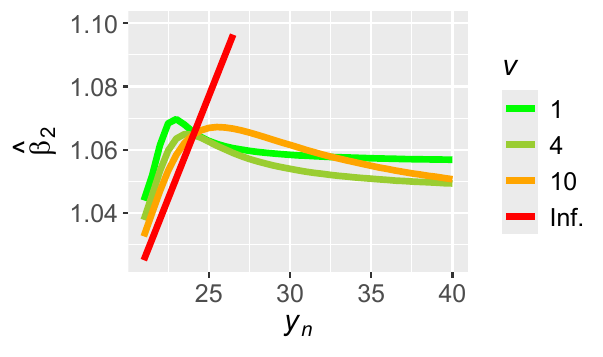}
  \vspace{-3mm}
\caption{Posterior mean of the slope coefficient $\beta_2$ as $y_{n}$ increases for the Student's $t$ linear regression with different degrees of freedom $\nu$, where $\nu = \text{Inf.}$ represents the normal linear regression.}\label{fig:mean}
\end{figure}

There is a rich frequentist literature on the topic of robustness against outliers, especially about linear regression, with celebrated works like that of \cite{huber1973robust} and \cite{beaton1974fitting} about the Huber and Tukey's biweight  M-estimators. The Bayesian literature is more sparse. A line of research in the Bayesian literature, called \textit{resolution of conflict} \citep{o2012bayesian}, aims to (mathematically) characterize the limiting behaviour of robust alternatives as outliers move further and further away from the bulk of the data, like the limiting behaviour observed in \autoref{fig:mean} for the Student's $t$ linear regression. The characterization is achieved by studying the limit of the associated posterior densities. Studying the limit of a posterior density is not easy due essentially to the presence of an integral in the denominator representing the marginal density evaluated at the observations or, equivalently, the normalizing constant.

First works in resolution of conflict focused on the location model (e.g., \cite{dawid1973posterior}, \cite{o1979outlier} and \cite{angers2007conflicting}) and the location--scale model (e.g., \cite{andrade2011bayesian} and \cite{desgagne2015robustness}). In recent years, the focus has been on linear regression \citep{DesGag2019, gagnon2020, gagnon2020PCR, 10.1214/22-BA1330, gagnon2023theoretical, hamura2020log, HAMURA2024110130, hamura2024short}, generalized linear models \citep{gagnon2023GLM, hamura2021robust}, Student's $t$ process \citep{https://doi.org/10.1111/sjos.12611} and accelerated failure time models \citep{hamura2025robust}. The proofs of the robustness characterization results are generally highly technical and not intuitive, limiting the accessibility and preventing the development of theory in that line of research. For instance, the first proof for the usual linear regression model, in \cite{gagnon2020}, involves the decomposition of the parameter space in mutually exclusive sets for which it is difficult to develop an intuition and which makes the majority of the steps in the proof technical, in addition to making the proof lengthy.

\cite{hamura2024short} recently highlighted the issue of accessibility. With the goal of improving accessibility of robustness characterization results and their proofs, the author presented a proof for a linear regression with a specific heavy-tailed error distribution. We however consider the attempt unsatisfactory as the heavy-tailed distribution assumed is not used in practice and, perhaps more importantly, the proof technique is the same as in \cite{gagnon2020} with the decomposition of the parameter space into mutually exclusive sets; the proof is thus not intuitive and highly technical. A goal of the current paper is to address the issue of accessibility in a way that is, in our opinion, more effective.

 With the current paper, we highlight that the undesired characteristics of the previous proofs are due to the generality of the assumptions on the prior distribution. The proposed approach is thus different than in \cite{hamura2024short}: we consider a specific class of prior distributions instead. In the paper, we focus on the typically used prior distribution in Bayesian normal linear regression given its conjugacy properties. This prior distribution is a conditional normal distribution for the regression coefficients and an inverse-gamma distribution for the squared scale of the errors. The framework is natural as it can be seen as that where a statistician is usually happy with the Bayesian normal linear regression (with this prior distribution), but this statistician worries that it may not be adapted for the current data set for which the presence of outliers is probable; thus the statistician wants to gain robustness and (only) changes the distribution assumption on the errors. By considering this specific prior distribution, we are able to present a robustness characterization result with a significantly simpler proof, to the extent that we are able to consider a remarkably general distribution assumption on the errors while keeping the proof simple; essentially, it can be any distribution with regularly varying or log-regularly varying tails. Note that there does not exist a result in such generality for models with regularly varying distributions. The conditional normal prior distribution allows to exploit a tail bound, representing the major difference with the previous proofs and the main reason of the significantly simpler proof. The proof is thus intuitive, uses classical probability arguments and is significantly shorter than the previous ones. To promote the development of theory in resolution of conflict, we highlight the key steps and explain how the proof can be adapted for another model. As an illustration, we present an application of the proof technique in the context of generalized linear models (GLMs).

  The main feature of the regression coefficient conditional prior PDF allowing the tail bound to be effective is its exponential tail decay. We thus show that the proof remains essentially unchanged and retains the same level of simplicity by instead assuming that each regression coefficient follows a sub-exponential distribution \citep[Section 2.7]{vershynin2018high}, without assuming that the distribution is the same for all the coefficients. The class of sub-exponential distributions include the Laplace distribution and sub-Gaussian distributions \citep[Section 2.5]{vershynin2018high}, meaning that it also includes the normal distribution. Regarding the error scale prior distribution, its main feature which contributes to the simplicity of the proof is having finite inverse moments. Therefore, we can use any distribution for positive random variables having finite inverse moments, like the log-normal distribution.

 We now describe how the rest of the paper is organized. In \autoref{sec:context}, we present in more detail the context, the model and the assumptions. In \autoref{sec:result}, we present the robustness characterization result and, in \autoref{sec:proof}, its proof. A conclusion follows in \autoref{sec:conclusion}.

 We finish this section with a general remark about robustness: there is of course a price to pay for a gain in robustness like that observed in \autoref{fig:mean} for the Student's $t$ linear regression. The price is twofold. Firstly, there is a loss in \textit{efficiency}, in the sense that, in the absence of outliers, the estimation is less efficient than with the benchmark (e.g., normal linear regression). The efficiency loss has been precisely measured for the Student's $t$ linear regression model in \cite{gagnon2023theoretical}. Secondly, there is an added computational complexity as all integrals need to be approximated using numerical methods, typically Markov chain Monte Carlo (MCMC) methods, even for linear regression. Hamiltonian Monte Carlo \citep{Duane1987, neal2011mcmc} has been used to approximate the posterior means for the Student's $t$ linear regression in \autoref{fig:mean}; see \autoref{sec:numerical} for more details. Variable selection and model averaging can be performed using a reversible jump algorithm \citep{green1995reversible, green2003trans}; see \cite{gagnon2019RJ}, \cite{gagnon2019NRJ} and \cite{gagnon2023} for efficient informed and non-reversible variants.

\section{Context, model and assumptions}\label{sec:context}

Let us assume that we have access to a data set of the form $\{\bx_i, y_i\}_{i=1}^n$, where $\bx_1, \ldots, \bx_n \in \re^p$ are $n \in \na$ vectors of explanatory variable data points and $y_1, \ldots, y_n \in \re$ are $n$ observations of a dependent variable, $p$ being a positive integer. Let us assume that one is interested in modelling the dependent variable through its relationship with the explanatory variables and, more specifically, in using a linear regression model. In such a model, it is assumed that $y_1, \ldots, y_n$ are realizations of $n$ random variables $Y_1, \ldots, Y_n$ defined as follows:
\begin{align}\label{eqn:linear_reg}
 Y_i = \bx_i^T \bbeta + \sigma \varepsilon_i, \quad i = 1, \ldots, n,
\end{align}
where $\bbeta =(\beta_1, \ldots, \beta_p)^T \in \re^p$ is a vector of regression coefficients, $\sigma > 0$ is a scale parameter and $\varepsilon_1, \ldots, \varepsilon_n \in \re$ are standardized errors. In an homoscedastic model, it is assumed that $\varepsilon_1, \ldots, \varepsilon_n$ are independent and identically distributed random variables, each having a PDF denoted here by $f$. In a Bayesian model, it is typically assumed that the two groups of random variables $(\varepsilon_1, \ldots, \varepsilon_n)$ and $(\bbeta, \sigma)$ are independent. The vectors $\bx_1, \ldots, \bx_n$ are thus typically considered to be fixed and known, that is not realizations of random variables, contrarily to $y_1, \ldots, y_n$.

We now present the assumptions on $f$.
\begin{Assumption}\label{ass:PDF}
   The PDF $f$ is strictly positive, symmetric and monotonic, that is $f(y) > 0$ and $f(y) = f(|y|)$ for all $y \in \re$, and, for any $|y_2| \geq |y_1|$, $f(|y_2|) \leq f(|y_1|)$. Also, it is bounded, that is there exists a constant $C > 0$ such that $f \leq C$. Finally, either of the following holds:
   \begin{enumerate}

\item[(i)] $f$ is a regularly varying function: there exist constants $C_f > 0$ and $\alpha > 0$ such that
    \[
     \lim_{y \rightarrow \pm \infty} \frac{f(y)}{C_f \, |y|^{-(\alpha + 1)}} = 1;
    \]

    \item [(ii)] $f$ is a log-regularly varying function: there exist constants $C_f > 0$ and $\alpha > 0$ such that
    \[
     \lim_{y \rightarrow \pm \infty} \frac{f(y)}{C_f \, |y|^{-1}(\log|y|)^{-(\alpha + 1)}} = 1.
    \]

   \end{enumerate}
\end{Assumption}
The first part of \autoref{ass:PDF} (positivity, symmetry, monotonicity and boundedness) represents regularity conditions which simplify the theoretical analysis. The second part is about tail thickness, heavy tails being essentially necessary for robustness. In this second part, we assume that $f$ is either regularly varying or log-regularly varying. Regularly varying functions have been extensively studied (see, e.g., \cite{resnick2007heavy} for a reference) and appear in many contexts, such as statistical network modelling \citep{caron2017sparse} and, of course, robustness against outliers as in the current paper. Note that we made an abuse of terminology in \autoref{ass:PDF} as the formal definition of a regularly varying function is slightly more general; we presented this version to simplify. We will nevertheless use the terminology ``regularly varying functions'' to refer to functions satisfying (i) in \autoref{ass:PDF}. As stated in \autoref{prop:Student}, the preferred PDF in robustness, the Student's $t$, satisfies \autoref{ass:PDF} as a regularly varying function.

\begin{Proposition}\label{prop:Student}
   Let $f$ be a Student's $t$ PDF with $\nu > 0$ degrees of freedom, that is
   \[
    f(y) = \frac{\Gamma\left(\frac{\nu + 1}{2}\right)}{\sqrt{\pi \nu} \, \Gamma\left(\frac{\nu}{2}\right)} \left(1 + \frac{y^2}{\nu}\right)^{-\frac{\nu + 1}{2}}, \quad y \in \re,
   \]
   where $\Gamma$ is the gamma function. Then, \autoref{ass:PDF} is satisfied.
\end{Proposition}

\begin{proof}
 It is straightforward to verify the first part of \autoref{ass:PDF}. It can be readily verified that $f$ is regularly varying using
 \[
  C_f = \frac{\Gamma\left(\frac{\nu + 1}{2}\right) \nu^{\nu/2}}{\sqrt{\pi} \, \Gamma\left(\frac{\nu}{2}\right)}
 \]
 and $\alpha = \nu$.
\end{proof}

The notion of log-regularly varying functions appeared more recently in \cite{desgagne2015robustness} in the context of robustness against outliers to achieve what is referred to as \textit{whole robustness} for the location--scale model (we will return to the concept of whole robustness in \autoref{sec:result}). As for regularly varying functions, we made an abuse of terminology in \autoref{ass:PDF} as the formal definition of a log-regularly varying function is slightly more general. We will nevertheless carry on with this abuse of terminology. An example of log-regularly varying PDFs is the \textit{log-Pareto-tailed normal} (LPTN). The central part of this continuous PDF coincides with the standard normal and the tails are log-Pareto, hence its name. It has an hyperparameter $\rho \in (2\Phi(1) - 1, 1) \approx (0.6827, 1)$ and is given by
\begin{align*}
f(y)=\left\{
\begin{array}{lcc}
                                                      \varphi(y)  & \text{ if } & |y|\leq \vartheta, \cr
                                                      \varphi(\vartheta)\,\frac{\vartheta}{|y|}\left(\frac{\log \vartheta}{\log |y|}\right)^{\lambda + 1} & \text{ if } & |y| > \vartheta, \cr
\end{array}
\right.
\end{align*}
where $\vartheta > 1$ and $\lambda > 0$ are functions of $\rho$ with
\begin{align*}
 & \vartheta=\Phi^{-1}((1+\rho)/2) = \{\vartheta : \Prob(-\vartheta \leq Z \leq \vartheta) = \rho \,\text{ for }\, Z\, \sim \, \mathcal{N}(0,1)\}, \cr
 & \lambda = 2(1-\rho)^{-1}\varphi(\vartheta) \, \vartheta \log(\vartheta),
 \end{align*}
$\varphi$ and $\Phi$ being the PDF and cumulative distribution function of a standard normal distribution, respectively.

\begin{Proposition}\label{prop:LPTN}
   Let $f$ be a LPTN PDF with $\rho \in (2\Phi(1) - 1, 1)$. Then, \autoref{ass:PDF} is satisfied.
\end{Proposition}

\begin{proof}
 It is straightforward to verify the first part of \autoref{ass:PDF}. It can be readily verified that $f$ is log-regularly varying using
 \[
  C_f = \varphi(\vartheta) \, \vartheta \, (\log \vartheta)^{\lambda + 1}
 \]
 and $\alpha = \lambda$.
\end{proof}

We now present the assumptions on the prior distribution.

\begin{Assumption}\label{ass:prior}
   The prior distribution is such that: $\bbeta$ given $\sigma$ has a normal distribution with a mean of $\mathbf{0}$ and a covariance matrix of $\sigma^2 \I_p$, where $\I_p$ is the identity matrix of size $p$, and $\sigma^2$ has an inverse-gamma distribution with any shape and scale parameters.
\end{Assumption}

As mentioned in \autoref{sec:intro}, this prior distribution is commonly used in Bayesian linear regression (see, e.g., \cite{raftery1997bayesian}). As also mentioned in \autoref{sec:intro}, there is a focus in the paper on this commonly used prior distribution as it is associated to a natural framework. In \autoref{sec:alternativeprior}, we show that our robustness characterization result holds for an important class of prior distributions, with essentially the same simple proof. The alternative to \autoref{ass:prior} considered is to assume that $\beta_1, \ldots, \beta_p$ and $\sigma$ are all independent, to simplify, and that each $\beta_j$ has a sub-exponential distribution (not necessarily with a mean of 0, not necessarily the same distribution for all $j$) and a probability distribution for $\sigma^2$ having finite inverse moments. The exponential decay of the regression coefficient prior PDF in \autoref{ass:prior} and \autoref{sec:alternativeprior} allows to exploit tail bounds. In \cite{gagnon2020} for instance, the class of prior distributions considered allows the decay to be much slower, but it prevents to exploit a tail bound, making the proof significantly longer and more technical.

With a (conditional) normal distribution on $\bbeta$ as in \autoref{ass:prior} (or a sub-exponential distribution as in \autoref{sec:alternativeprior}), one has to be careful with the potential conflict between the prior information and that carried by the data \citep{10.1214/22-BA1330}. Note that this is true also for the prior distribution on $\sigma^2$ in \autoref{ass:prior} given that the inverse-gamma PDF has a thin left tail. Ideally, the scale parameter of the inverse-gamma would be of the same order of magnitude as $\sigma^2$ to mitigate the risk of conflicting prior information. A small value for the shape parameter makes the inverse-gamma PDF relatively flat and thus yields a (joint) prior distribution that is as weakly informative as possible for the class of prior distributions in \autoref{ass:prior}.

\section{Robustness characterization result}\label{sec:result}

To characterize the robustness of the model in \eqref{eqn:linear_reg} (depending on $f$), we study it under an asymptotic framework where the outliers move further and further away from the bulk of the data. We mathematically represent this asymptotic framework by considering that the outliers move along particular paths (as in, e.g., \cite{gagnon2020} and \cite{hamura2020log}). The mathematical representation allows for a general definition of outliers, that is couples $(\bx_i, y_i)$ whose components are incompatible with the trend in the bulk of the data. Let us consider for example that there is an element in $\bx_i$ that makes the combination of $\bx_i$ with $y_i$ incompatible. Equivalently, in this example, we can consider that, compared with the trend in the bulk of the data, the value of $y_i$ is either too small or too large for this $\bx_i$. We can thus allow for this general definition of outliers by considering an asymptotic framework where the vectors $\bx_i$ are fixed (but potentially extreme) and the observations $y_i$ are such that
\[
 |y_i| = a_i + b_i \omega, \quad i =1, \ldots, n,
\]
with $a_i > 0$ a constant, $b_i = 0$ if the data point is a non-outlier and $b_i \geq 1$ if it is an outlier, and then we let $\omega \rightarrow \infty$.

Under such an asymptotic framework, we obtain a sequence of posterior distributions, indexed by $\omega$, and we want to understand what a posterior distribution in this sequence looks like when $\omega$ is large. Note that this asymptotic framework is not in contradiction with the assumption in \autoref{sec:context} that $y_1, \ldots, y_n$ are realizations of the random variables $Y_1, \ldots, Y_n$ with the model in \eqref{eqn:linear_reg}. Indeed, all observation values $y_1, \ldots, y_n$ are possible under this model, but they become less probable as the values become more extreme. The asymptotic framework thus allows to study how the posterior distribution behaves when some observations (the outlying observations) become more and more extreme.

We will prove a theoretical asymptotic result characterizing the limiting posterior distribution which implies that, for the outlying data points with $\bx_i$ fixed (but potentially extreme), there exist large enough values for $|y_i| = a_i + b_i \omega$ such that the associated posterior distribution is similar to the limiting one. The location of the point $(\bx_i, y_i)$ has an impact on how large $|y_i|$ needs to be; for instance, it needs to be larger when $\bx_i$ is extreme, justifying the use of different $a_i$ and $b_i$ for the different points. For a real data set with (fixed) outliers, the goal of this mathematical representation is to be able to choose values for all $a_i$ and $b_i$ and a value for $\omega$ so that this data set is obtained.

We now present definitions that will allow to state the robustness characterization result. Let us define the index set of outlying data points by: $\Oset := \{i: b_i \geq 1\}$. The index set of non-outlying data points is thus given by: $\Oset^\mathsf{c} := \{1, \ldots, n\} \setminus \Oset$. We also define the set of non-outlying observations: $\by_{\Oset^\mathsf{c}} := \{y_i: i \in \Oset^\mathsf{c}\}$. Let us denote by $\pi$ the prior distribution of $(\bbeta, \sigma)$. We consider that it is not in conflict with the trend in the bulk of the data to focus on robustness against outliers; see \cite{10.1214/22-BA1330} for a study of robustness of heavy-tailed prior distributions against conflicting prior information in regression. Let us denote by $\pi_\omega(\, \cdot \,, \cdot \mid \by)$ a posterior distribution in the sequence indexed by $\omega$, with a posterior density, denoted by $\pi_\omega(\, \cdot \,, \cdot \mid \by)$ as well to simplify, which is such that
\begin{align}\label{eqn:post}
 \pi_\omega(\bbeta, \sigma \mid \mathbf{y}) = \pi(\bbeta, \sigma) \left[\prod_{i = 1}^n (1 / \sigma) f((y_i - \bx_i^T \bbeta) / \sigma)\right] \Bigg/ m_\omega(\by), \quad \bbeta \in \re^p, \sigma > 0,
\end{align}
where $\by = (y_1, \ldots, y_n)^T$ and
\begin{align}\label{eqn:constant}
 m_\omega(\by) =  \int_{\re^p} \int_0^\infty \pi(\bbeta, \sigma) \left[\prod_{i = 1}^n (1 / \sigma) f((y_i - \bx_i^T \bbeta) / \sigma)\right] \, \d\sigma \, \d\bbeta,
\end{align}
if $m_\omega(\by) < \infty$, a situation where the posterior distribution is proper and thus well defined.

From \eqref{eqn:post}, we understand that the limiting behaviour of the (conditional) PDF of $Y_i$ evaluated at an outlying point is central to the characterization of the robustness properties of a robust alternative. We now present a proposition about this limiting behaviour.

\begin{Proposition}\label{prop:limit_PDF}
   Suppose that \autoref{ass:PDF} holds. For all $\bbeta \in \re^p$ and $\sigma > 0$,
    \[
     \lim_{y_i \rightarrow \pm\infty} \frac{(1 / \sigma) f((y_i - \bx_i^T \bbeta) / \sigma)}{g(\sigma) f(y_i)} = 1,
    \]
    where $g(\sigma) = \sigma^\alpha$ if $f$ is a regularly varying function or $g(\sigma) = 1$ if $f$ is a log-regularly varying function.
\end{Proposition}

\begin{proof}
 Let us first consider the case where $f$ is regularly varying. For all $\bbeta \in \re^p$ and $\sigma > 0$,
     \begin{align*}
     &\lim_{y_i \rightarrow \pm\infty} \frac{(1 / \sigma) f((y_i - \bx_i^T \bbeta) / \sigma)}{f(y_i)}\cr
      &\quad =  \lim_{y_i \rightarrow \pm\infty} \frac{f((y_i - \bx_i^T \bbeta) / \sigma)}{C_f \, |(y_i - \bx_i^T \bbeta) / \sigma|^{-(\alpha + 1)}}\frac{C_f \, |y_i|^{-(\alpha + 1)}}{f(y_i)} \frac{(1 / \sigma) C_f \, |(y_i - \bx_i^T \bbeta) / \sigma|^{-(\alpha + 1)}}{C_f \, |y_i|^{-(\alpha + 1)}} = \sigma^\alpha.
    \end{align*}

    Let us now consider the case where $f$ is log-regularly varying. For all $\bbeta \in \re^p$ and $\sigma > 0$,
     \begin{align*}
     &\lim_{y_i \rightarrow \pm\infty} \frac{(1 / \sigma) f((y_i - \bx_i^T \bbeta) / \sigma)}{f(y_i)} \cr
      &\quad=  \lim_{y_i \rightarrow \pm\infty} \frac{f((y_i - \bx_i^T \bbeta) / \sigma)}{C_f \, |(y_i - \bx_i^T \bbeta) / \sigma|^{-1}(\log|(y_i - \bx_i^T \bbeta) / \sigma|)^{-(\alpha + 1)}}\frac{C_f \, |y_i|^{-1}(\log|y_i|)^{-(\alpha + 1)}}{f(y_i)} \cr
      &\hspace{80mm} \times \frac{ C_f \, |y_i - \bx_i^T \bbeta|^{-1}(\log|(y_i - \bx_i^T \bbeta) / \sigma|)^{-(\alpha + 1)}}{C_f \, |y_i|^{-1}(\log|y_i|)^{-(\alpha + 1)}} = 1.
    \end{align*}
\end{proof}

Under \autoref{ass:PDF}, we thus expect the PDF term of each outlier in \eqref{eqn:post} to behave like $g(\sigma) f(y_i) \propto g(\sigma)$ asymptotically. The result that we present below is specifically about this. We prove convergence of the posterior distribution towards $\pi(\, \cdot \,, \cdot \mid \by_{\Oset^\mathsf{c}})$ which is such that
\[
 \pi(\bbeta, \sigma \mid \mathbf{y}_{\Oset^\mathsf{c}}) = \pi(\boldsymbol\beta, \sigma) \, g(\sigma)^{|\Oset|} \left[\prod_{i \in \Oset^\mathsf{c}} (1 / \sigma) f((y_i - \mathbf{x}_i^T \boldsymbol\beta) / \sigma)\right] \Bigg/ m(\mathbf{y}_{\Oset^\mathsf{c}}), \quad \boldsymbol\beta \in \re^p, \sigma > 0,
\]
where $|\Oset|$ is the cardinality of the set $\Oset$, that is the number of outliers, and
\begin{align}\label{eqn:constant_wo}
 m(\mathbf{y}_{\Oset^\mathsf{c}}) =  \int_{\re^p} \int_0^\infty \pi(\boldsymbol\beta, \sigma) \, g(\sigma)^{|\Oset|}  \left[\prod_{i \in \Oset^\mathsf{c}} (1 / \sigma) f((y_i - \mathbf{x}_i^T \boldsymbol\beta) / \sigma)\right] \, \d\sigma \, \d\boldsymbol\beta,
\end{align}
if $m(\mathbf{y}_{\Oset^\mathsf{c}}) < \infty$, a situation where the limiting posterior distribution is proper and thus well defined. Note that we abused notation by writing, for instance, $\pi(\, \cdot \,, \cdot \mid \by_{\Oset^\mathsf{c}})$ as the latter is not the conditional distribution given only the non-outliers in the case where $f$ is regularly varying; there is an additional term, $g(\sigma)^{|\Oset|} = \sigma^{|\Oset| \alpha}$, in the definitions above.

When $f$ is log-regularly varying, there is asymptotically no trace of the outliers in the posterior distribution as $g(\sigma) = 1$. The robust alternative thus acts automatically like practitioners would and excludes the outliers when they are far enough from the bulk of the data and there is no doubt as to whether they really are outliers. Such a robust alternative is said to achieve whole robustness. When $f$ is regularly varying, there is asymptotically a trace of the outliers in the posterior distribution, namely $g(\sigma) = \sigma^\alpha$ for each outlier. It is nevertheless possible to obtain a limit which, by definition, does not depend on $\omega$, the latter representing in a sense the source of outlyingness. Such a robust alternative is thus said to achieve \textit{partial robustness}. Note that, typically, for both cases of regularly and log-regularly varying models, the impact of observations gradually diminish when they are artificially moved away from the bulk of the data, as observed in \autoref{fig:mean}. Moderately far observations thus have a certain influence, reflecting uncertainty about the nature of these observations in a grey zone (outliers versus non-outliers).

Based on these characteristics of regularly varying and log-regularly varying PDFs, a recommendation is to use a linear regression model with a log-regularly varying PDF on the errors given its whole robustness property (see \cite{gagnon2020} for a detailed treatment of the subject). A LPTN distribution can for instance be assumed as the error distribution. A disadvantage of the LPTN PDF is that, while being equal to the normal PDF in the area where the mass concentrates and thus globally similar to the normal PDF, it is not smooth (its first derivative is not continuous). This may make less efficient typical MCMC methods. This disadvantage is not shared by the Student's $t$ distribution, which can be additionally represented as a scale mixture of normal distributions. A Gibbs sampler \citep{geman1984stochastic} can thus be implemented, leading to a simplified computational procedure. It has also been observed that the difference in robustness with the LPTN linear regression model is not significant in some situations when the degrees of freedom of the Student's $t$ distribution are small, say $\nu = \alpha = 4$ (see, e.g., \cite{gagnon2020}). A user can thus weigh the pros and cons and take an informed decision regarding which robust alternative to use.

In order to state the robustness characterization result, we need a guarantee that all posterior distributions are well defined ($\pi_\omega(\, \cdot \,, \cdot \mid \by)$ and $\pi(\, \cdot \,, \cdot \mid \by_{\Oset^\mathsf{c}})$). We present an assumption on the number of outliers $|\Oset|$, or equivalently the number of non-outliers $|\Oset^\mathsf{c}| = n - |\Oset|$, that allows such a guarantee.

\begin{Assumption}\label{ass:sample_size}
   In the case where $f$ is regularly varying, the assumption is that $|\Oset^\mathsf{c}| > \alpha |\Oset| \Leftrightarrow |\Oset| / n < 1 / (\alpha + 1)$. In the case where $f$ is log-regularly varying, the assumption is that $|\Oset^\mathsf{c}| \geq |\Oset| \Leftrightarrow |\Oset| / n \leq 1 / 2$.
\end{Assumption}

\begin{Proposition}\label{prop:proper}
   Suppose that Assumptions \ref{ass:PDF}, \ref{ass:prior} and \ref{ass:sample_size} hold.  Then, $m(\mathbf{y}_{\Oset^\mathsf{c}}) < \infty$ and $m_\omega(\by) < \infty$ for all $\omega$.
\end{Proposition}

\begin{proof}
 We prove the result for the case where $f$ is regularly varying; the proof for the case where $f$ is log-regularly varying is similar. When $f$ is regularly varying,
\begin{align*}
 m(\mathbf{y}_{\Oset^\mathsf{c}}) &=  \int_{\re^p} \int_0^\infty \pi(\boldsymbol\beta, \sigma) \, \sigma^{\alpha|\Oset|}  \left[\prod_{i \in \Oset^\mathsf{c}} (1 / \sigma) f((y_i - \mathbf{x}_i^T \boldsymbol\beta) / \sigma)\right] \, \d\sigma \, \d\boldsymbol\beta \cr
 &\leq C^{|\Oset^\mathsf{c}|} \int_{\re^p} \int_0^\infty \pi(\boldsymbol\beta, \sigma) \, \sigma^{\alpha|\Oset| - |\Oset^\mathsf{c}|}  \, \d\sigma \, \d\boldsymbol\beta = C^{|\Oset^\mathsf{c}|} \E[\sigma^{-(|\Oset^\mathsf{c}| - \alpha|\Oset|)}] < \infty,
\end{align*}
using that $f \leq C$ (\autoref{ass:PDF}) in the first inequality and, in the final inequality, that $\E[(\sigma^2)^{-\kappa}] < \infty$ for any $\kappa > 0$ when $\sigma^2$ has an inverse-gamma distribution (\autoref{ass:prior}), given that $|\Oset^\mathsf{c}| > \alpha |\Oset|$ (\autoref{ass:sample_size}).

Also,
\begin{align*}
   m_\omega(\by) &=  \int_{\re^p} \int_0^\infty \pi(\bbeta, \sigma) \left[\prod_{i = 1}^n (1 / \sigma) f((y_i - \bx_i^T \bbeta) / \sigma)\right] \, \d\sigma \, \d\bbeta \cr
   &\leq C^n \int_{\re^p} \int_0^\infty \pi(\bbeta, \sigma) \, \sigma^{-n} \, \d\sigma \, \d\bbeta < \infty,
\end{align*}
using, similarly, \autoref{ass:PDF} in the first inequality and \autoref{ass:prior} in the final one.
\end{proof}

Note that the notion of linear regression is not used in the proof of \autoref{prop:proper}; the proof is valid for any model as long as $f$, the conditional PDF of $Y_i$, is bounded and the prior distribution satisfies regularity conditions. Note also that the result of \autoref{prop:proper} can actually be obtained without \autoref{ass:sample_size} in the case where $f$ is log-regularly varying. \autoref{ass:sample_size} is, in this case, used for the robustness characterization result. We now present this result.

\begin{Theorem}\label{Thm}
Suppose that Assumptions \ref{ass:PDF}, \ref{ass:prior} and \ref{ass:sample_size} hold. As $\omega \rightarrow \infty$,
 \begin{description}
   \item[(a)] the asymptotic behaviour of the marginal distribution is: $m_\omega(\by) / \prod_{i \in \Oset}f(y_i)\rightarrow m(\by_{\Oset^\mathsf{c}})$;

     \item[(b)] the posterior density converges pointwise: for any $\bbeta \in \re^p$ and $\sigma > 0$, $\pi_\omega(\bbeta, \sigma \mid \by) \rightarrow \pi(\bbeta, \sigma \mid \by_{\Oset^\mathsf{c}})$;

\item[(c)] the posterior distribution converges: $\pi_\omega(\, \cdot \,, \, \cdot \mid \by) \rightarrow \pi(\, \cdot \,, \, \cdot \mid \by_{\Oset^\mathsf{c}})$.
\end{description}
\end{Theorem}

An appealing aspect of \autoref{Thm} (which is typical of recent robustness characterization results) is the simplicity of the assumptions. They are easy to understand. Assumptions \ref{ass:PDF} and \ref{ass:prior} are about choices made by the model user (the error PDF $f$ and prior distribution), who is thus able to assess that these assumptions are satisfied. \autoref{ass:sample_size} is expected to hold, at least when $f$ is log-regularly varying. \autoref{ass:sample_size} is about the proportion of outliers $|\Oset| / n$ in the data set and is related to the notion of \textit{breakdown point}, generally defined as the proportion of outliers that an estimator can handle. \autoref{ass:sample_size} suggests that it is $1 / (\alpha + 1)$ in the case where $f$ is regularly varying. In this case, the validity of the assumption can be evaluated based on prior knowledge (the proportion of outliers expected for a given data set) or using outlier detection (see \cite{gagnon2020} for a technique in the context of Bayesian linear regression).

At this point, it is natural to ask whether \autoref{ass:sample_size} is necessary for the case where $f$ is regularly varying (it seems uninteresting to ask the question for the case where $f$ is log-regularly varying because we only require the proportion of outliers in this case to be less than $50\%$, corresponding to the usually desired bound). We performed a numerical experiment suggesting that \autoref{ass:sample_size} is (essentially) necessary for the case where $f$ is regularly varying. The experiment is the same as that described in \autoref{sec:intro}, except that we increased the value of more than one $y_i$. The results are presented in \autoref{fig:mean2}. In \autoref{fig:mean2} (a), the results are for the case where two observations, $y_{n - 1}$ and $y_n$, are gradually increased, with $y_{n - 1} = y_n$. In this plot, we observe a different behaviour than in \autoref{fig:mean} for the Student's $t$ model with $\nu = \alpha = 10$. In this case, $|\Oset| / n = 1 / 10$ is not lesser than $1 / (\alpha + 1) = 1 / 11$, but it is close. In fact, \autoref{ass:sample_size} can be refined to include the shape parameter of the inverse-gamma distribution of $\sigma^2$. Let us denote this shape parameter by $a>0$. When $f$ is regularly varying, \autoref{ass:sample_size} can be stated with $(|\Oset^\mathsf{c}| - \alpha |\Oset|) / 2 + a > 0$ instead. This is for the convergence in distribution (\autoref{Thm}). Because we estimate the parameter using the posterior mean, what we in fact require is $(|\Oset^\mathsf{c}| - \alpha |\Oset|) / 2 + a > 1$ (we will return to this below). In our numerical experiment, $a = 2$, which implies that $(|\Oset^\mathsf{c}| - \alpha |\Oset|) / 2 + a = 1$ which is not greater than 1 but is equal to it. There is thus a violation of the condition but it is not significant, which provides an explanation for the different convergence behaviour observed in \autoref{fig:mean2} (a). In \autoref{fig:mean2} (b), the last three observations, $y_{n -2}, y_{n - 1}$ and $y_n$, are gradually increased, with $y_{n -2} = y_{n - 1} = y_n$. In this case, $(|\Oset^\mathsf{c}| - \alpha |\Oset|) / 2 + a = -5.5$, which is significantly smaller than 1, and the estimate for the Student's $t$ model with $\nu = \alpha = 10$ increases similarly as that for the normal model, showing non-robustness. Our numerical experiment thus suggests that \autoref{ass:sample_size} is (essentially) necessary.

\begin{figure}[ht]
\centering
$\begin{array}{cc}
    \vspace{-2mm}  \hspace{-0mm} \includegraphics[width=0.50\textwidth]{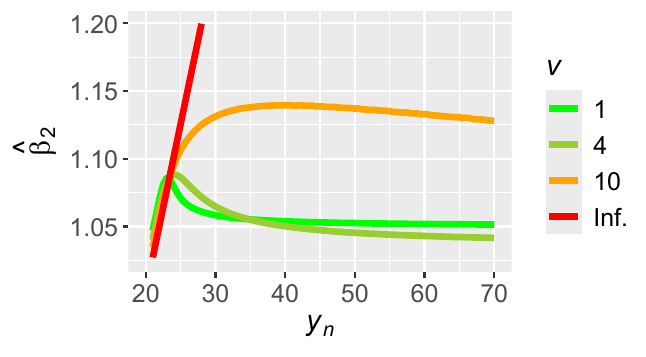} & \hspace{-5mm} \includegraphics[width=0.50\textwidth]{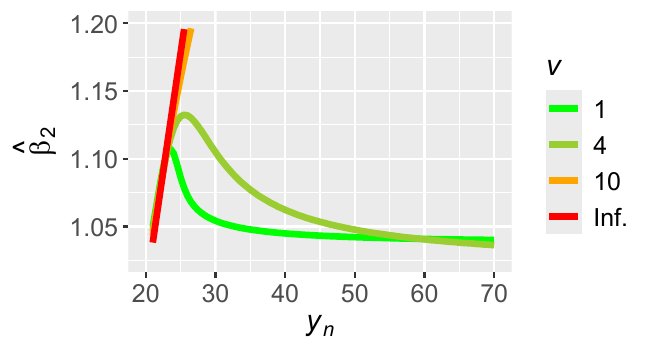} \cr
    \hspace{-0mm} \textbf{(a) 2 observations go to $+\infty$} & \hspace{-5mm} \textbf{(b) 3 observations go to $+\infty$}
\end{array}$
  \vspace{-2mm}
\caption{Posterior mean of the slope coefficient $\beta_2$ when increasing (a) $y_{n - 1}$ and $y_{n}$ with $y_{n - 1} = y_{n}$ and (b) $y_{n - 2}, y_{n - 1}$ and $y_{n}$ with $y_{n - 2} = y_{n - 1} = y_{n}$, for the Student's $t$ linear regression with different degrees of freedom $\nu$, where $\nu = \text{Inf.}$ represents the normal linear regression.}\label{fig:mean2}
\end{figure}

Regarding \autoref{ass:PDF}, the first part about regularity conditions on $f$ (positivity, symmetry, monotonicity and boundedness) is, as mentioned in \autoref{sec:context}, not necessary, but it simplifies the proofs. The second part about tail heaviness is essentially necessary. Indeed, it is necessary to have the limit in \autoref{prop:limit_PDF} to obtain \autoref{Thm}, and using a regularly or a log-regularly function $f$ is essentially necessary to have the limit in \autoref{prop:limit_PDF}. The assumption about the prior distribution, \autoref{ass:prior}, is not necessary, but it simplifies the proofs. As mentioned in \autoref{sec:context}, an alternative to \autoref{ass:prior}, for which \autoref{Thm} holds and the proof is as simple, is that where each regression coefficient has a sub-exponential prior distribution and $\sigma^2$ has a prior distribution having finite inverse moments (see \autoref{sec:alternativeprior}).

Let us now discuss the results in \autoref{Thm}. Result (a) is the centrepiece; it is the result that allows to obtain relatively easily Results (b) and (c), the latter being the interesting and important results. It states that $m_\omega(\by)$ is asymptotically equivalent to $m(\by_{\Oset^\mathsf{c}}) \prod_{i \in \Oset}f(y_i)$ (recall \eqref{eqn:constant} and \eqref{eqn:constant_wo}). Its demonstration requires considerable work as it is about the characterization of the part of the posterior density with an integral (the result is essentially that we are allowed to interchange the limit and the integral and to use \autoref{prop:limit_PDF}). Result (b) ensures the convergence of the maximum a posteriori estimate and thus that the latter is robust, if the estimate always remains within a compact subset of the parameter space as $\omega \rightarrow \infty$ and if the estimate based on $\pi(\bbeta, \sigma \mid \by_{\Oset^\mathsf{c}})$ belongs to a compact subset of the parameter space. Result (c) indicates that any estimation of $\bbeta$ and $\sigma$ based on posterior quantiles (e.g., using posterior medians or Bayesian credible intervals) is robust to outliers. It is also possible to ensure the convergence of moments under more technical assumptions (see \cite{gagnon2020}) and thus that moments are robust. All these results characterize the limiting behaviour of a variety of Bayes estimators. Finally, note that in variable selection, when the joint posterior distribution of the models and parameters is considered, this joint distribution converges if the prior distributions of the parameters of all models satisfy \autoref{ass:prior} (or the alternative assumption presented in \autoref{sec:alternativeprior}).

\section{Proof of \autoref{Thm}}\label{sec:proof}

 We start with the proof of Result (c) (assuming Result (b)). Next, we prove Result (b) (assuming Result (a)). Finally, we provide the proof of Result (a), which is longer.  In the proof, we highlight the key steps and explain how they can be adapted for another model. As an example, we present an application of the proof technique in the context of GLMs in \autoref{sec:robustGLM}. For the proof of \autoref{Thm}, we assume that $|\Oset| \geq 1$, meaning that there is at least one outlier; otherwise, the proof is trivial.

 Result (c) is a direct consequence of Result (b) by Scheffé's lemma, which states that the pointwise convergence of a PDF is sufficient to ensure the convergence in distribution (see \cite{scheffe1947useful}). To prove Result (b), we rewrite $\pi_\omega(\bbeta, \sigma \mid \by)$ for fixed $\bbeta \in \re^p$ and $\sigma > 0$ in order to exploit Result (a) and \autoref{prop:limit_PDF}:
        \[
 \pi_\omega(\bbeta, \sigma \mid \by) = \pi(\bbeta,\sigma\mid\by_{\Oset^\mathsf{c}}) \, \frac{m(\by_{\Oset^\mathsf{c}}) \prod_{i \in \Oset}f(y_i)}{m_\omega(\by)} \prod_{i \in \Oset } \frac{(1/\sigma)f((y_i-\bx_i^T\bbeta)/\sigma)}{g(\sigma) f(y_i)}.
\]
Note that $m(\mathbf{y}_{\Oset^\mathsf{c}}) < \infty$ and $m_\omega(\by) < \infty$ for all $\omega$ under Assumptions \ref{ass:PDF}, \ref{ass:prior} and \ref{ass:sample_size} (see \autoref{prop:proper}). For any $\bbeta \in \re^p$ and $\sigma > 0$,
\[
 \frac{m(\by_{\Oset^\mathsf{c}}) \prod_{i \in \Oset}f(y_i)}{m_\omega(\by)} \prod_{i \in \Oset } \frac{(1/\sigma)f((y_i-\bx_i^T \bbeta)/\sigma)}{g(\sigma) f(y_i)} \rightarrow 1,
\]
by Result (a) and \autoref{prop:limit_PDF}. The proof of Result (b) exploits the notion of linear regression but only through the limit in \autoref{prop:limit_PDF}. If the model was different, the limit result would take another form and the proof of Result (b) would need to be adapted accordingly.

We now prove Result (a) by showing that
\[
 \frac{m_\omega(\by)}{m(\by_{\Oset^\mathsf{c}})\prod_{i \in \Oset} f(y_i)} \rightarrow 1.
\]
As mentioned, this result is more difficult to prove because it involves a limit of integrals. More precisely, we need to evaluate the limit of the numerator above, in which $\prod_{i \in \Oset} f(y_i)$ is included (because it depends on $\omega$); recall \eqref{eqn:constant} and that $|y_i| = a_i + b_i \omega$ with $b_i \geq 1$ for $i \in \Oset$). To evaluate this limit, we in fact combine the numerator and the denominator in the expression above and obtain an integral involving the same expression as in \autoref{prop:limit_PDF}:
\begin{align*}
\frac{m_\omega(\by)}{m(\by_{\Oset^\mathsf{c}})\prod_{i \in \Oset} f(y_i)} &= \frac{m_\omega(\by)}{m(\by_{\Oset^\mathsf{c}})\prod_{i \in \Oset} f(y_i)}
      \int_{\re^p}\int_{0}^{\infty}\pi_\omega(\bbeta, \sigma \mid \by) \, \d\sigma \, \d\bbeta \nonumber \\
   &=   \int_{\re^p}\int_{0}^{\infty}\frac{\pi(\bbeta, \sigma) \prod_{i=1}^{n}
        (1/\sigma)f((y_i-\bx_i^T\bbeta)/\sigma)}{m(\by_{\Oset^\mathsf{c}}) \prod_{i \in \Oset} f(y_i)} \, \d\sigma \, \d\bbeta \nonumber \\
   &=   \int_{\re^p}\int_{0}^{\infty}
        \pi(\bbeta, \sigma\mid \by_{\Oset^\mathsf{c}})  \prod_{i \in \Oset } \frac{(1/\sigma)f((y_i-\bx_i^T\bbeta)/\sigma)}{g(\sigma) f(y_i)} \, \d\sigma \, \d\bbeta =: I(\omega).
\end{align*}
By \autoref{prop:limit_PDF}, we would obtain the result, that is $\lim_{\omega \rightarrow \infty} I(\omega) = 1$, if we were allowed to interchange the limit and the integral. We essentially prove that we are allowed to do so. Note that, again, this part of the proof exploits the notion of linear regression only through the limit in \autoref{prop:limit_PDF}. If the model was different, the limit result would take another form and this part would need to be adapted accordingly.

The form of $I(\omega)$ suggests the use of results like Lebesgue’s dominated convergence theorem to prove Result (a). In the case where $(y_i - \bx_i^T \bbeta)/\sigma$ is of the order of $\omega$ for $i \in \Oset$, we expect to be able to bound
\[
 \frac{(1/\sigma)f((y_i - \bx_i^T \bbeta)/\sigma)}{g(\sigma) f(y_i)}
\]
in a way that it does not depend on $\omega$ given the form of the tails of $f$ (\autoref{ass:PDF}); recall that $y_i$ is of the order of $\omega$ for $i \in \Oset$. To make this concrete, we can think of the case where $f$ is regularly varying and thus the tails have a polynomial form. We follow this strategy and define a set for $\bbeta$ on which it is (essentially) guaranteed that $(y_i - \bx_i^T \bbeta)/\sigma$ is of the order of $\omega$:
         \[
          S(\omega) := \bigcap_{i=1}^n \left\{\bbeta: |\bx_i^T\bbeta| \leq \omega / 2\right\}.
         \]
         The definition of this set exploits the notion of linear regression to (essentially) obtain that $(y_i - \bx_i^T \bbeta)/\sigma$ is of the order of $\omega$; if the model was different, the definition would need to be adapted accordingly.

         We write
         \[
          I(\omega) = I_1(\omega) + I_2(\omega),
         \]
 where
            \[
 I_1(\omega) = \int_{\re^p}\int_{0}^{\infty} \ind_{S(\omega)} \,
        \pi(\bbeta, \sigma\mid \by_{\Oset^\mathsf{c}})  \prod_{i \in \Oset } \frac{(1/\sigma)f((y_i - \bx_i^T \bbeta)/\sigma)}{g(\sigma) f(y_i)} \, \d\sigma \, \d\bbeta,
            \]
            and $I_2(\omega)$ is the integral on $S(\omega)^\mathsf{c}$. Note that $\ind_{S(\omega)} \rightarrow \ind_{\re^p}$ as $\omega \rightarrow \infty$ given that, for any $\bbeta \in \re^p$, there exists $\omega$ large enough so that $|\mathbf{x}_i^T\boldsymbol\beta| \leq \omega / 2$ for all $i$.

            We now prove that, on $S(\omega) \times (0, \infty)$, the integrand in $I(\omega)$ is bounded by $\pi(\bbeta, \sigma)$ times a polynomial in $1 / \sigma$, which does not depend on $\omega$ and is integrable (under \autoref{ass:prior}). This implies that $\lim_{\omega \rightarrow \infty} I_1(\omega) = 1$ by Lebesgue’s dominated convergence theorem (and \autoref{prop:limit_PDF}). Next, on $S(\omega)^\mathsf{c} \times (0, \infty)$, we exploit the (prior) conditional normality of $\bbeta$ to prove that $\lim_{\omega \rightarrow \infty} I_2(\omega) = 0$, which will allow to conclude that $\lim_{\omega \rightarrow \infty} I(\omega) = 1$.

            In the proof of Result (a), we thus use a decomposition of the parameter space into mutually exclusive sets, in a way, like in \cite{gagnon2020}. There is however an important difference as the sets are not the same; in the current framework, we can easily develop an intuition for the introduction of these sets and these sets do not make the majority of the steps in the proof technical and the proof lengthy.

            For $\bbeta \in S(\omega)$ and $\sigma > 0$,
\begin{align}
&\pi(\bbeta, \sigma\mid \by_{\Oset^\mathsf{c}})  \prod_{i \in \Oset } \frac{(1/\sigma)f((y_i - \bx_i^T \bbeta)/\sigma)}{g(\sigma) f(y_i)} \cr
&\quad \propto \pi(\bbeta, \sigma) \, g(\sigma)^{|\Oset|} \prod_{i \in \text{O}^\mathsf{c}} (1 / \sigma) f((y_i - \bx_i^T \bbeta) / \sigma) \prod_{i \in \Oset } \frac{(1/\sigma)f((y_i-\bx_i^T\bbeta)/\sigma)}{g(\sigma) f(y_i)} \cr
&\quad \leq C^{|\Oset^\mathsf{c}|} \, \pi(\bbeta, \sigma) \, g(\sigma)^{|\Oset|} \, \frac{1}{\sigma^{|\Oset^\mathsf{c}|}} \prod_{i \in \Oset } \frac{(1/\sigma) f((y_i - \bx_i^T \bbeta)/\sigma)}{g(\sigma) f(y_i)}  \cr
&\quad \leq C^{|\Oset^\mathsf{c}|} \, \pi(\bbeta, \sigma) \, g(\sigma)^{|\Oset|} \, \frac{1}{\sigma^{|\Oset^\mathsf{c}|}} \prod_{i \in \Oset } \frac{(1/\sigma) f(\omega/(2\sigma))}{g(\sigma) f(2b_i\omega)}  \cr
&\quad \leq C^{|\Oset^\mathsf{c}|} \, \pi(\bbeta, \sigma) \, C_2 \left(\frac{1}{\sigma^{\kappa}} + 1\right), \label{eqn:boundpart1proof}
\end{align}
using in the second line that $m(\by_{\Oset^\mathsf{c}})$ is a finite constant (\autoref{prop:proper}), in the third line that $f \leq C$ (\autoref{ass:PDF}), in the fourth line the monotonicity of $f$ (\autoref{ass:PDF}; more details follow), and \autoref{lemma:1} in the last line, $C_2$ and $\kappa$ being two positive constants independent of $\bbeta, \sigma$ and $\omega$. About the fourth line, for the numerator in
\[
 \frac{(1/\sigma) f((y_i - \bx_i^T \bbeta)/\sigma)}{g(\sigma) f(y_i)},
\]
we used that, for $i \in \Oset$ and $\bbeta \in S(\omega)$, $|y_i - \bx_i^T \bbeta| / \sigma \geq ||y_i| - |\bx_i^T \bbeta|| / \sigma \geq (a_i + b_i \omega - \omega / 2) / \sigma \geq \omega / (2 \sigma)$ by the reverse triangle inequality and because  $a_i > 0$ and $b_i \geq 1$. For the denominator, we used that $|y_i| = a_i + b_i \omega \leq 2 b_i \omega$ for large enough $\omega$. This part of the proof thus exploits the notion of linear regression to obtain the bound $|y_i - \bx_i^T \bbeta| / \sigma \geq \omega / (2 \sigma)$; if the model was different, it would need to be adapted accordingly. \autoref{lemma:1} is a technical lemma which makes precise the bound obtained for
\[
 \frac{(1/\sigma) f(\omega/(2\sigma))}{g(\sigma) f(2b_i\omega)}
\]
based on the tails of $f$ (\autoref{ass:PDF}). We easily see that, when $f$ is regularly varying, the $\omega$'s in the numerator and denominator cancel each other out given the polynomial form of the tails and we thus obtain a bound which is a function of $\sigma$. Under \autoref{ass:prior}, $(\sigma^2)^{-\kappa / 2}$ in the bound in \eqref{eqn:boundpart1proof} is integrable with respect to $\pi(\boldsymbol\beta, \sigma)$ (because $\sigma^2$ has an inverse-gamma distribution). Thus, by Lebesgue’s dominated convergence theorem and \autoref{prop:limit_PDF}, $\lim_{\omega \rightarrow \infty} I_1(\omega) = 1$.

We now turn to proving that $\lim_{\omega \rightarrow \infty} I_2(\omega) = 0$. We have that
\begin{align*}
&\int_{\re^p} \int_0^\infty  \ind_{S(\omega)^\mathsf{c}} \, \pi(\bbeta, \sigma\mid \by_{\Oset^\mathsf{c}})  \prod_{i \in \Oset } \frac{(1/\sigma)f((y_i - \bx_i^T \bbeta)/\sigma)}{g(\sigma) f(y_i)} \, \d \sigma \, \d\bbeta \cr
&\quad \propto \int_{\re^p} \int_0^\infty  \ind_{S(\omega)^\mathsf{c}} \, \pi(\bbeta, \sigma)  \, g(\sigma)^{|\Oset|}  \prod_{i \in \text{O}^\mathsf{c}} (1 / \sigma) f((y_i - \bx_i^T \bbeta) / \sigma) \prod_{i \in \Oset } \frac{(1/\sigma)f((y_i - \bx_i^T \bbeta)/\sigma)}{g(\sigma) f(y_i)} \, \d \sigma \, \d\bbeta \cr
&\quad \leq C^n \int_{\re^p} \int_0^\infty  \ind_{S(\omega)^\mathsf{c}} \, \pi(\bbeta, \sigma) \, \frac{1}{\sigma^n} \prod_{i \in \Oset } \frac{1}{f(y_i)} \, \d \sigma \, \d\bbeta \cr
&\quad \leq C^n \int_{\re^p} \int_0^\infty  \ind_{S(\omega)^\mathsf{c}} \, \pi(\bbeta, \sigma) \, \frac{1}{\sigma^n} \prod_{i \in \Oset } \frac{1}{f(2b_i\omega)} \, \d \sigma \, \d\bbeta \cr
&\quad\propto\left(\prod_{i \in \Oset } \frac{1}{f(2b_i\omega)}\right) \E\left[\sigma^{-n}  \Prob\left(\bigcup_{i=1}^n \left\{\bbeta: |\bx_i^T\bbeta| > \omega / 2\right\} \mid \sigma\right)\right],
\end{align*}
using in the second line that $m(\by_{\Oset^\mathsf{c}})$ is a finite constant (\autoref{prop:proper}), in the third line that $f \leq C$ (\autoref{ass:PDF}) and the monotonicity of $f$ in the fourth line (\autoref{ass:PDF}) given that $|y_i| = a_i + b_i \omega \leq 2 b_i \omega$ for large enough $\omega$. Notice how this part of the proof does not exploit the notion of linear regression, except for the definition of $S(\omega)^\mathsf{c}$ which appears in the probability in the last line.

We finish the proof by showing that
\[
 \E\left[\sigma^{-n}  \Prob\left(\bigcup_{i=1}^n \left\{\bbeta: |\bx_i^T\bbeta| > \omega / 2\right\} \mid \sigma\right)\right]
\]
goes to 0 more quickly than $\prod_{i \in \Oset } f(2b_i\omega)^{-1}$ goes to infinity. If $\bbeta$ and $\sigma$ were independent \textit{a priori} (with a prior covariance proportional to $\I_p$ for $\bbeta$), we would have that $\Prob\left(\bigcup_{i=1}^n \left\{\bbeta: |\bx_i^T\bbeta| > \omega / 2\right\}\right)$ go to 0 exponentially quickly given that $\bx_i^T\bbeta$ is normal. Because $\prod_{i \in \Oset } f(2b_i\omega)^{-1}$ goes to infinity polynomially quickly (\autoref{ass:PDF}), we would be able to conclude. This is used in \autoref{sec:alternativeprior} to prove that \autoref{Thm} holds with essentially the same proof by assuming that each regression coefficient has a sub-exponential prior distribution and $\sigma^2$ has a prior distribution having finite inverse moments.

 Here, we need to be more careful because $\bbeta$ and $\sigma$ are not independent \textit{a priori}:
\begin{align*}
\E\left[\sigma^{-n} \Prob\left(\bigcup_{i=1}^n \left\{\bbeta: |\bx_i^T\bbeta| > \omega / 2\right\} \mid \sigma\right)\right] &\leq \sum_{i=1}^n  \E\left[\sigma^{-n} \Prob\left(\left\{\bbeta: |\bx_i^T\bbeta| > \omega / 2\right\} \mid \sigma\right)\right] \cr
 &\leq \sum_{i=1}^n\E\left[\sigma^{-n}  \frac{1}{\sqrt{2\pi}} \frac{4 \|\bx_i\| \sigma}{\omega} \, \exp\left(-\frac{\omega^2}{8 \|\bx_i\|^2 \sigma^2}\right)\right],
\end{align*}
using in the first inequality the union bound and in the second inequality that, given $\sigma$, $\bx_i^T\bbeta$ has a normal distribution with a mean of $0$ and a variance of $\|\bx_i\|^2 \sigma^2$, together with the following tail bound: for $Z_{\sigma_0} \sim \mathcal{N}(0, \sigma_0^2)$ with $\sigma_0 > 0$ a constant,
\[
 \Prob(Z_{\sigma_0} \geq t) \leq \frac{1}{\sqrt{2 \pi}} \frac{\sigma_0}{t} \exp\left(-\frac{t^2}{2\sigma_0^2}\right), \quad t > 0.
\]
This tail bound is relatively well known, but we provide a proof for completeness in \autoref{sec:lemmas} (see \autoref{lemma:3}).

We now prove that
\[
  \left[\prod_{i \in \Oset } f(2b_i\omega)^{-1}\right]\E\left[\sigma^{-n}  \frac{1}{\sqrt{2\pi}} \frac{4 \|\bx_i\| \sigma}{\omega} \, \exp\left(-\frac{\omega^2}{8 \|\bx_i\|^2 \sigma^2}\right)\right] \rightarrow 0,
\]
for all $i$, which will allow to conclude. We omit the constants (with respect to $\omega$) to simplify as they do not change the conclusion. Under \autoref{ass:prior}, $\tau = \sigma^{-2}$ has a gamma distribution and let us denote by $a >0$ and $b >0$ its scale and shape parameters, respectively. We have that
\begin{align*}
 &\left[\prod_{i \in \Oset } f(2b_i\omega)^{-1}\right] \frac{1}{\omega} \E\left[\sigma^{-(n - 1)} \exp\left(-\frac{\omega^2}{8 \|\bx_i\|^2 \sigma^2}\right)\right] \cr
 &\quad= \left[\prod_{i \in \Oset } f(2b_i\omega)^{-1}\right] \frac{1}{\omega} \E\left[\tau^{(n - 1) / 2} \exp\left(-\frac{\omega^2 \tau}{8 \|\bx_i\|^2}\right)\right] \cr
 &\quad= \left[\prod_{i \in \Oset } f(2b_i\omega)^{-1}\right] \frac{1}{\omega} \int_0^\infty \tau^{(n - 1) / 2} \exp\left(-\frac{\omega^2 \tau}{8 \|\bx_i\|^2}\right) \frac{\tau^{a - 1} \exp(-\tau / b)}{\Gamma(a) \, b^a} \, \d\tau \cr
 &\quad=\left[\prod_{i \in \Oset } f(2b_i\omega)^{-1}\right] \frac{1}{\omega}\frac{\Gamma((n-1)/2 + a) \, \left(\frac{1}{b} + \frac{\omega^2}{8 \|\bx_i\|^2}\right)^{-((n-1)/2 + a)}}{\Gamma(a) \, b^a} \int_0^\infty \frac{\tau^{(n - 1) / 2 + a - 1} \exp\left(-\tau \Big/\left(\frac{1}{b} + \frac{\omega^2}{8 \|\bx_i\|^2}\right)^{-1}\right)}{\Gamma((n-1)/2 + a) \, \left(\frac{1}{b} + \frac{\omega^2}{8 \|\bx_i\|^2}\right)^{-((n-1)/2 + a)}} \, \d\tau \cr
 &\quad \leq\left[\prod_{i \in \Oset } f(2b_i\omega)^{-1}\right] \frac{\Gamma((n-1)/2 + a) \, (8 \|\bx_i\|^2)^{(n-1)/2+a}}{\Gamma(a) \, b^a}\frac{1}{\omega^{n + 2a}} \cr
  &\quad\leq \left[\prod_{i \in \Oset } f(2b_i\omega)^{-1}\right] \frac{\Gamma((n-1)/2 + a) \, (8 \|\bx_i\|^2)^{(n-1)/2+a}}{\Gamma(a) \, b^a}\frac{1}{\omega^{n}},
\end{align*}
using in the first inequality that $1/b > 0$ and in the second one that $a > 0$. The integral in the fourth line is equal to 1 as it is the integral of a gamma PDF over the whole support.

The proof is concluded given that
\[
 \frac{1}{\omega^{n}} \prod_{i \in \Oset } f(2b_i\omega)^{-1} \rightarrow 0
\]
as $\omega\rightarrow \infty$ by \autoref{lemma:2}. This lemma is technical and makes precise the convergence. Essentially, when $f$ is regularly varying, $f(2b_i\omega)^{-1}$ is of the order of $\omega^{\alpha + 1}$ and the convergence is obtained as $\prod_{i \in \Oset } \omega^{\alpha + 1} = \omega^{\alpha |\Oset| + |\Oset|}$, $n =  |\Oset^\mathsf{c}| +  |\Oset|$ and $|\Oset^\mathsf{c}| > \alpha |\Oset|$ (\autoref{ass:sample_size}).

\section{Conclusion}\label{sec:conclusion}

In this paper, we focused on promoting the development of theory in \emph{resolution of conflict} and the accessibility of this line of research  within the Bayesian literature about robustness against outliers. To promote the accessibility, we presented a simple and intuitive proof of a robustness characterization result for a general heavy-tailed linear regression model. The key element making the proof simple and intuitive, while considering a broad class of heavy-tailed error distributions, is the exponential decay of the regression coefficient prior PDF and the finiteness of the inverse moments of the error scale prior distribution. To promote the development of theory, we highlighted in the proof the key steps that would need to be adapted when proving a robustness characterization result for a different model. We present an application of the proof technique in the context of GLMs in \autoref{sec:robustGLM}, clarifying how to adapt these steps.

By focusing on the line of research of resolution of conflict, we did not discuss broadly in this paper the different approaches to robustness against outliers. The approach covered in this paper consists of using heavy-tailed distributions, which is a classical approach in Bayesian statistics. The distribution used typically depends on the model that one wants to make robust. Thus, the approach is not generic. There do exists generic approaches, such as that of \cite{ghosh2016robust} which consists of using a density power divergence. It is discussed more broadly in the context of divergence-based loss functions in \cite{jewson2018principles}. Both works can be seen as fitting within the generalized Bayesian framework of \cite{bissiri2016general}. Recently, \cite{bhatia2024bayesian} proposed a different generic approach which is instead based on a robust MCMC scheme. With this approach, the robustness comes from the algorithm which is used for inference. The strength of these methods lies in their generality, at the price of being less transparent and preventing the derivation of precise results. It is the opposite for the classical approach of using heavy-tailed distributions.

\bibliographystyle{rss}
\bibliography{references}

\appendix

\section{Numerical experiment}\label{sec:numerical}

The numerical experiment whose results are presented in \autoref{fig:mean} is based on an analysis of a simulated data set with $n = 20$, $p = 2$, $(x_{1,2}, x_{2,2}, \ldots, x_{n,2}) = (1, 2, \ldots, n)$, and where $y_1, \ldots, y_n$ were first sampled using intercept and slope coefficients both equal to 1, an error scaling of 1 and errors sampled independently from the standard normal distribution; we then obtain a sequence of data sets by gradually increasing the value of $y_n$.

To estimate the parameters of each Student's $t$ model, we sample from the posterior distribution using Hamiltonian Monte Carlo (HMC). To run this algorithm, we need to evaluate the posterior density up to a normalizing constant and to evaluate the gradient of the log density. We now write the posterior density (up to a normalizing constant), and next,  the gradient of the log density. Let us consider that the shape and scale parameters of the inverse-gamma prior distribution are $a > 0$ and $b >0$, respectively. We write the posterior density by considering $\tau :=\sigma^2$ as the variable:
\begin{align*}
 \pi_\omega(\bbeta, \tau \mid \by) &\propto \pi(\tau) \, \pi(\bbeta \mid \tau) \prod_{i=1}^n \frac{1}{\tau^{1/2}}f\left(\frac{y_i - \mathbf{x}_i^T \boldsymbol\beta}{\tau^{1/2}}\right) \cr
 &= \pi(\tau) \, \pi(\bbeta \mid \tau) \, \frac{1}{\tau^{n/2}} \prod_{i=1}^n f\left(\frac{y_i - \mathbf{x}_i^T \boldsymbol\beta}{\tau^{1/2}}\right).
\end{align*}

For typical MCMC samplers (such as HMC), it is usually good practice to apply changes of variables to obtain variables that all take values on the real line. We thus define $\gamma := \log \tau$ and obtain
\[
\pi_\omega(\bbeta, \gamma \mid \by) \propto \pi(\e^{\gamma}) \, \pi(\bbeta \mid \e^{\gamma}) \, \frac{1}{\e^{\gamma (n  / 2 - 1)}} \prod_{i=1}^n f\left(\frac{y_i - \mathbf{x}_i^T \boldsymbol\beta}{\e^{\gamma / 2}}\right).
\]
The log density is such that (if we forget about the normalizing constant):
\[
\log \pi_\omega(\bbeta, \gamma \mid \by) = \log \pi(\e^{\gamma}) + \log \pi(\bbeta \mid \e^{\gamma}) - (n / 2 - 1) \gamma + \sum_{i=1}^n \log f\left(\frac{y_i - \mathbf{x}_i^T \boldsymbol\beta}{\e^{\gamma / 2}}\right),
\]
where, under \autoref{ass:prior} and for the Student's $t$ model, $\log \pi(\e^{\gamma})$ is the log PDF of the inverse-gamma distribution evaluated at $\e^{\gamma}$, $\log \pi(\bbeta \mid \e^{\gamma})$ is the log PDF of a normal distribution with a mean of $\mathbf{0}$ and a covariance matrix of $\e^{\gamma} \I_p$ evaluated at $\bbeta$ and $\log f$ is the log PDF of a Student's $t$ distribution with $\nu$ degrees of freedom. The gradient is thus such that:
\[
 \frac{\partial}{\partial \bbeta} \, \log \pi_\omega(\bbeta, \gamma \mid \by) = - \e^{-\gamma} \bbeta +\e^{-\gamma} \, \frac{\nu + 1}{\nu} \sum_{i = 1}^n \left(1 + \frac{(y_i - \mathbf{x}_i^T \boldsymbol\beta)^2}{\e^{\gamma} \nu}\right)^{-1} (y_i - \mathbf{x}_i^T \boldsymbol\beta) \, \mathbf{x}_i,
\]
\begin{align*}
\frac{\partial}{\partial \gamma} \, \log \pi_\omega(\bbeta, \gamma \mid \by) &= -(a + 1) + b \e^{-\gamma} +  \frac{\e^{-\gamma}}{2} \, \bbeta^T \bbeta - (n / 2 + p / 2 - 1) \cr
&\qquad  + \frac{\nu + 1}{2\nu} \, \e^{-\gamma} \sum_{i = 1}^n  \left(1 + \frac{(y_i - \mathbf{x}_i^T \boldsymbol\beta)^2}{\e^{\gamma} \nu}\right)^{-1} (y_i - \mathbf{x}_i^T \boldsymbol\beta)^2.
\end{align*}

For the numerical experiment, we also need to compute the posterior expectation of the slope coefficient $\beta_2$ for the normal model. We now present a proposition with an explicit expression for this expectation.

\begin{Proposition}\label{prop:post_norm}
   Suppose that \autoref{ass:prior} holds and that the shape and scale parameters of the inverse-gamma are $a > 0$ and $b >0$, respectively. If $f$ is the standard normal PDF, then the posterior distribution is such that: $\bbeta$ given $\sigma$ has a normal distribution with a mean of $\hat{\bbeta}$ and a covariance matrix of $\sigma^2(\bX^T \bX + \I_p)^{-1}$, and $\sigma^2$ has an inverse-gamma distribution with a shape parameter of $(2a + n) / 2$ and a scale parameter of
   \[
 \frac{2b + \by^T \by  - \hat{\bbeta}^T (\bX^T \bX + \I_p) \hat{\bbeta}}{2},
\]
 where $\hat{\bbeta} = (\bX^T \bX + \I_p)^{-1} \bX^T \by$ and $\bX$ is the design matrix. In particular, the posterior expectation of $\bbeta$ is $\hat{\bbeta}$.
\end{Proposition}

\begin{proof}
 We write the proof by considering $\tau =\sigma^2$ as the variable. In normal linear regression, $\bY$, given $\bbeta$ and $\tau$, has a normal distribution with a mean of $\bX \bbeta$ and a covariance matrix of $\tau \I_n$. Therefore, we can write the posterior density as:
 \begin{align*}
  \pi_\omega(\bbeta, \tau \mid \by) &\propto \pi(\tau) \, \frac{1}{\tau^{p/2}} \exp\left(-\frac{1}{2\tau} \, \bbeta^T \bbeta\right) \frac{1}{\tau^{n/2}} \exp\left(-\frac{1}{2\tau} (\by - \bX \bbeta)^T (\by - \bX \bbeta)\right) \cr
  &= \pi(\tau) \, \frac{1}{\tau^{\frac{p+n}{2}}} \exp\left(-\frac{1}{2\tau}\left[ (\by - \bX \bbeta)^T (\by - \bX \bbeta) + \bbeta^T \bbeta\right] \right).
 \end{align*}
 We analyse the term in the exponential:
 \begin{align*}
  (\by - \bX \bbeta)^T (\by - \bX \bbeta) + \bbeta^T \bbeta &= \by^T \by - \by^T \bX \bbeta - (\bX \bbeta)^T \by + \bbeta^T \bX^T \bX \bbeta + \bbeta^T \bbeta \cr
  &= \by^T \by - \by^T \bX \bbeta - (\bX \bbeta)^T \by + (\bbeta -  \hat{\bbeta} +  \hat{\bbeta})^T (\bX^T \bX + \I_p) (\bbeta -  \hat{\bbeta} +  \hat{\bbeta}) \cr
  &= \by^T \by + (\bbeta -  \hat{\bbeta})^T (\bX^T \bX + \I_p) (\bbeta -  \hat{\bbeta}) - \hat{\bbeta}^T (\bX^T \bX + \I_p) \hat{\bbeta},
 \end{align*}
 using that $\by^T \bX \bbeta = (\bX \bbeta)^T \by$ (because it is a scalar) and
 \begin{align*}
 \hat{\bbeta}^T (\bX^T \bX + \I_p) \bbeta = \bbeta^T   (\bX^T \bX + \I_p) \hat{\bbeta} =\bbeta^T   (\bX^T \bX + \I_p) (\bX^T \bX + \I_p)^{-1} \bX^T \by =  (\bX \bbeta)^T \by.
 \end{align*}
 Therefore,
 \begin{align*}
  \pi_\omega(\bbeta, \tau \mid \by) &\propto \pi(\tau) \, \frac{1}{\tau^{\frac{n}{2}}} \exp\left(-\frac{1}{2\tau}\left[\by^T \by  - \hat{\bbeta}^T (\bX^T \bX + \I_p) \hat{\bbeta}\right] \right) \frac{1}{\tau^{\frac{p}{2}}} \exp\left(-\frac{1}{2\tau} (\bbeta -  \hat{\bbeta})^T (\bX^T \bX + \I_p) (\bbeta -  \hat{\bbeta}) \right).
 \end{align*}
From this, we can conclude that $\bbeta$ given $\tau$ has a normal distribution with a mean of $\hat{\bbeta}$ and a covariance matrix of $\tau (\bX^T \bX + \I_p)^{-1}$. Regarding $\tau$, we have that
\begin{align*}
 \pi(\tau) \, \frac{1}{\tau^{\frac{n}{2}}} \exp\left(-\frac{1}{2\tau}\left[\by^T \by  - \hat{\bbeta}^T (\bX^T \bX + \I_p) \hat{\bbeta}\right] \right) \propto \frac{1}{\tau^{\frac{2a + n}{2} + 1}} \exp\left(- \frac{1}{2\tau}\left[2b + \by^T \by  - \hat{\bbeta}^T (\bX^T \bX + \I_p) \hat{\bbeta}\right]\right),
\end{align*}
which allows to conclude that the posterior distribution of $\tau$ is an inverse-gamma with a shape parameter of $(2a + n) / 2$ and a scale parameter of
\[
 \frac{2b + \by^T \by  - \hat{\bbeta}^T (\bX^T \bX + \I_p) \hat{\bbeta}}{2}.
\]
\end{proof}

\section{Three lemmas}\label{sec:lemmas}

In this section, we present three lemmas used in the proof of \autoref{Thm}.

\begin{Lemma}\label{lemma:1}
Suppose Assumptions \ref{ass:PDF} and \ref{ass:sample_size} hold. For all $\omega$ large enough and $\sigma > 0$, there exist constants $C_2 > 0$ and $\kappa > 0$, such that
\[
 g(\sigma)^{|\Oset|} \, \frac{1}{\sigma^{|\Oset^\mathsf{c}|}} \prod_{i \in \Oset } \frac{(1/\sigma) f(\omega/(2\sigma))}{g(\sigma) f(2b_i\omega)} \leq C_2 \left(\frac{1}{\sigma^\kappa} + 1\right),
\]
the constants $C_2 > 0$ and $\kappa > 0$ being thus independent of $\omega$ and $\sigma$.
\end{Lemma}

\begin{proof}
 First, we prove the result for the case where $f$ is regularly varying. In this case,
\[
 g(\sigma)^{|\Oset|} \, \frac{1}{\sigma^{|\Oset^\mathsf{c}|}} \prod_{i \in \Oset } \frac{(1/\sigma) f(\omega/(2\sigma))}{g(\sigma) f(2b_i\omega)} = \frac{1}{\sigma^{|\Oset^\mathsf{c}| - |\Oset|\alpha}} \prod_{i \in \Oset } \frac{(1/\sigma) f(\omega/(2\sigma))}{\sigma^\alpha f(2b_i\omega)}.
\]
From \autoref{ass:PDF}, we can deduce that for all $0 < \delta < 1$, there exists $\y0 >0$ such that for all $|y| > \y0$,
\[
 (1- \delta) C_f \, |y|^{-(\alpha + 1)} < f(y) < (1 + \delta) C_f \, |y|^{-(\alpha + 1)}.
\]
Let us consider such a $\delta$. For large enough $\omega$, $2b_i\omega \geq \y0$, and therefore,
\[
 \frac{1}{\sigma^{|\Oset^\mathsf{c}| - |\Oset|\alpha}} \prod_{i \in \Oset } \frac{(1/\sigma) f(\omega/(2\sigma))}{\sigma^\alpha f(2b_i\omega)}  \leq (1 - \delta)^{-|\Oset|} \frac{1}{\sigma^{|\Oset^\mathsf{c}| - |\Oset|\alpha}} \prod_{i \in \Oset } \frac{(1/\sigma) f(\omega/(2\sigma))}{\sigma^\alpha C_f \, (2b_i\omega)^{-(\alpha + 1)}}.
\]
Now, we consider two situations. First, we consider that $\omega/(2\sigma) > \y0$. In this situation,
\begin{align*}
(1 - \delta)^{-|\Oset|} \frac{1}{\sigma^{|\Oset^\mathsf{c}| - |\Oset|\alpha}} \prod_{i \in \Oset } \frac{(1/\sigma) f(\omega/(2\sigma))}{\sigma^\alpha C_f \, (2b_i\omega)^{-(\alpha + 1)}} &\leq (1 + \delta)^{|\Oset|}  (1 - \delta)^{-|\Oset|} \frac{1}{\sigma^{|\Oset^\mathsf{c}| - |\Oset|\alpha}} \prod_{i \in \Oset } \frac{(1/\sigma) C_f \, (\omega/(2\sigma))^{-(\alpha + 1) }}{\sigma^\alpha C_f \, (2b_i\omega)^{-(\alpha + 1)}} \cr
&= (1 + \delta)^{|\Oset|}  (1 - \delta)^{-|\Oset|} \frac{1}{\sigma^{|\Oset^\mathsf{c}| - |\Oset|\alpha}} \prod_{i \in \Oset } (4 b_i)^{\alpha + 1}.
\end{align*}
Second, we consider that $\omega/(2\sigma) < \y0 \Leftrightarrow 1 / \sigma < 2\y0 / \omega$. In this situation,
\begin{align*}
(1 - \delta)^{-|\Oset|} \frac{1}{\sigma^{|\Oset^\mathsf{c}| - |\Oset|\alpha}} \prod_{i \in \Oset } \frac{(1/\sigma) f(\omega/(2\sigma))}{\sigma^\alpha C_f \, (2b_i\omega)^{-(\alpha + 1)}} &\leq   (1 - \delta)^{-|\Oset|} \frac{1}{\sigma^{|\Oset^\mathsf{c}| - |\Oset|\alpha}} \prod_{i \in \Oset } \frac{(2 \y0)^{\alpha+1} C}{\omega^{\alpha + 1} C_f \, (2b_i\omega)^{-(\alpha + 1)}} \cr
&=  (1 - \delta)^{-|\Oset|} \frac{1}{\sigma^{|\Oset^\mathsf{c}| - |\Oset|\alpha}} \prod_{i \in \Oset } (4 b_i \y0)^{\alpha + 1}  C / C_f,
\end{align*}
using that $f \leq C$ (\autoref{ass:PDF}). Therefore, in both situations, there exists a constant $C_2 > 0$ such that
\[
  g(\sigma)^{|\Oset|} \, \frac{1}{\sigma^{|\Oset^\mathsf{c}|}} \prod_{i \in \Oset } \frac{(1/\sigma) f(\omega/(2\sigma))}{g(\sigma) f(2b_i\omega)} \leq C_2 \frac{1}{\sigma^{|\Oset^\mathsf{c}| - |\Oset|\alpha}} \leq C_2 \left(\frac{1}{\sigma^{|\Oset^\mathsf{c}| - |\Oset|\alpha}} + 1\right).
\]

Now, we prove the result for the case where $f$ is log-regularly varying. The proof is similar. In this case,
\[
 g(\sigma)^{|\Oset|} \, \frac{1}{\sigma^{|\Oset^\mathsf{c}|}} \prod_{i \in \Oset } \frac{(1/\sigma) f(\omega/(2\sigma))}{g(\sigma) f(2b_i\omega)} = \frac{1}{\sigma^{|\Oset^\mathsf{c}|}} \prod_{i \in \Oset } \frac{(1/\sigma) f(\omega/(2\sigma))}{ f(2b_i\omega)}.
\]
From \autoref{ass:PDF}, we can deduce that for all $0 < \delta < 1$, there exists $\y0 >0$ such that for all $|y| > \y0$,
\[
 (1- \delta) C_f \, |y|^{-1}(\log|y|)^{-(\alpha + 1)} < f(y) < (1 + \delta) C_f \, |y|^{-1}(\log|y|)^{-(\alpha + 1)}.
\]
Let us consider such a $\delta$. For large enough $\omega$, $2b_i\omega \geq \y0$, and therefore,
\[
 \frac{1}{\sigma^{|\Oset^\mathsf{c}| }} \prod_{i \in \Oset } \frac{(1/\sigma) f(\omega/(2\sigma))}{ f(2b_i\omega)}  \leq (1 - \delta)^{-|\Oset|} \frac{1}{\sigma^{|\Oset^\mathsf{c}| }} \prod_{i \in \Oset } \frac{(1/\sigma) f(\omega/(2\sigma))}{ C_f \, (2b_i\omega)^{-1}(\log(2b_i\omega))^{-(\alpha + 1)}}.
\]
Now, we consider two situations. First, we consider that $\omega/(2\sigma) \geq \omega^{1/4} > \y0$ (for large enough $\omega$). In this situation,
\begin{align*}
&(1 - \delta)^{-|\Oset|} \frac{1}{\sigma^{|\Oset^\mathsf{c}| }} \prod_{i \in \Oset } \frac{(1/\sigma) f(\omega/(2\sigma))}{ C_f \, (2b_i\omega)^{-1}(\log(2b_i\omega))^{-(\alpha + 1)}}\cr
 & \quad \leq (1 + \delta)^{|\Oset|} (1 - \delta)^{-|\Oset|} \frac{1}{\sigma^{|\Oset^\mathsf{c}| }} \prod_{i \in \Oset } \frac{(1/\sigma) C_f \, (\omega / (2\sigma))^{-1}(\log(\omega / (2\sigma)))^{-(\alpha + 1)}}{ C_f \, (2b_i\omega)^{-1}(\log(2b_i\omega))^{-(\alpha + 1)}} \cr
& \quad =   (1 + \delta)^{|\Oset|}  (1 - \delta)^{-|\Oset|} \frac{1}{\sigma^{|\Oset^\mathsf{c}| }} \prod_{i \in \Oset } (4 b_i)^{\alpha + 1} \left(\frac{1 + \frac{\log(2b_i)}{\log\omega}}{1 - \frac{\log(2\sigma)}{\log \omega}}\right)^{\alpha + 1} \cr
& \quad \leq   (1 + \delta)^{|\Oset|}  (1 - \delta)^{-|\Oset|} \frac{1}{\sigma^{|\Oset^\mathsf{c}| }} \prod_{i \in \Oset } (4 b_i)^{\alpha + 1} \left(\frac{1 + \frac{\log(2b_i)}{\log\omega}}{1 / 4}\right)^{\alpha + 1},
\end{align*}
using that $\log(2\sigma) \leq \log(\omega^{3/4}) = (3/4) \log(\omega)$. All terms in the final bound, except $1 / \sigma^{|\Oset^\mathsf{c}| }$, are constant with respect to $\sigma$ and bounded with respect to $\omega$. Second, we consider that $\omega/(2\sigma) < \omega^{1/4} \Leftrightarrow 1 / \sigma < 2 / \omega^{3/4} \leq 1$ (for large enough $\omega$). In this situation,
\begin{align*}
(1 - \delta)^{-|\Oset|} \frac{1}{\sigma^{|\Oset^\mathsf{c}| }} \prod_{i \in \Oset } \frac{(1/\sigma) f(\omega/(2\sigma))}{ C_f \, (2b_i\omega)^{-1}(\log(2b_i\omega))^{-(\alpha + 1)}} &\leq(1 - \delta)^{-|\Oset|} \frac{1}{\sigma^{|\Oset^\mathsf{c}|  - |\Oset|}} \prod_{i \in \Oset } \frac{4 C}{\omega^{3/2} C_f \, (2b_i\omega)^{-1}(\log(2b_i\omega))^{-(\alpha + 1)}} \cr
&\leq (1 - \delta)^{-|\Oset|}  \prod_{i \in \Oset } 8b_i (C / C_f) \frac{(\log(2b_i\omega))^{\alpha + 1}}{\omega^{1/2}},
\end{align*}
using that $f \leq C$ (under \autoref{ass:PDF}) and $1 / \sigma^2 \leq 4 /\omega^{3/2}$ in the first inequality, and that $\sigma^{-(|\Oset^\mathsf{c}|  - |\Oset|)} \leq 1$ given that $|\Oset^\mathsf{c}|  - |\Oset|\geq 0$ (\autoref{ass:sample_size}). All terms in the final bound are constant with respect to $\sigma$ and bounded with respect to $\omega$. Therefore, in both situations, there exists a constant $C_2 > 0$ such that
\[
  g(\sigma)^{|\Oset|} \, \frac{1}{\sigma^{|\Oset^\mathsf{c}|}} \prod_{i \in \Oset } \frac{(1/\sigma) f(\omega/(2\sigma))}{g(\sigma) f(2b_i\omega)} \leq C_2 \left(\frac{1}{\sigma^{|\Oset^\mathsf{c}| }} + 1\right).
\]
\end{proof}

\begin{Lemma}\label{lemma:2}
Suppose Assumptions \ref{ass:PDF} and \ref{ass:sample_size} hold. As $\omega\rightarrow \infty$,
\[
 \frac{1}{\omega^{n}} \prod_{i \in \Oset } f(2b_i\omega)^{-1} \rightarrow 0.
\]
\end{Lemma}

\begin{proof}
  First, we prove the result for the case where $f$ is regularly varying. As shown in the proof of \autoref{lemma:1},
\[
\frac{1}{\omega^{n}} \prod_{i \in \Oset } \frac{1}{f(2b_i\omega)}  \leq (1 - \delta)^{-|\Oset|} \frac{1}{\omega^{n}} \prod_{i \in \Oset } \frac{1}{C_f \, (2b_i\omega)^{-(\alpha + 1)}} = (1 - \delta)^{-|\Oset|} C_f^{-|\Oset|}\left(\prod_{i \in \Oset} (2b_i)^{\alpha+1}\right) \frac{1}{\omega^{|\Oset^\mathsf{c}| - |\Oset|\alpha}} \rightarrow 0,
\]
using that $n = |\Oset^\mathsf{c}| + |\Oset|$ and $|\Oset^\mathsf{c}| > \alpha |\Oset|$ (\autoref{ass:sample_size}).

Now, we prove the result for the case where $f$ is log-regularly varying. The proof is similar. Again, as shown in the proof of \autoref{lemma:1},
\begin{align*}
\frac{1}{\omega^{n}} \prod_{i \in \Oset } \frac{1}{f(2b_i\omega)}  &\leq (1 - \delta)^{-|\Oset|} \frac{1}{\omega^{n}} \prod_{i \in \Oset } \frac{1}{C_f \, (2b_i\omega)^{-1}(\log(2b_i\omega))^{-(\alpha + 1)}} \cr
 &= (1 - \delta)^{-|\Oset|} (2 C_f)^{-|\Oset|}\left(\prod_{i \in \Oset} b_i\right) \frac{\prod_{i \in \Oset} (\log(2b_i\omega))^{\alpha + 1}}{\omega^{|\Oset^\mathsf{c}| }} \rightarrow 0,
\end{align*}
using that $n = |\Oset^\mathsf{c}| + |\Oset|$ and $|\Oset^\mathsf{c}| \geq |\Oset| \geq 1$ (\autoref{ass:sample_size}).
\end{proof}

\begin{Lemma}\label{lemma:3}
For $Z_{\sigma_0} \sim \mathcal{N}(0, \sigma_0^2)$ with $\sigma_0 > 0$ a constant,
\[
 \Prob(Z_{\sigma_0} \geq t) \leq \frac{1}{\sqrt{2 \pi}} \frac{\sigma_0}{t} \exp\left(-\frac{t^2}{2\sigma_0^2}\right), \quad t > 0.
\]
\end{Lemma}

\begin{proof}
 We have that
 \[
  \Prob(Z_{\sigma_0} \geq t) = \Prob(Z \geq t / \sigma_0),
 \]
 where $Z \sim \mathcal{N}(0, 1)$. We prove that
 \[
 \Prob(Z \geq z) \leq \frac{1}{\sqrt{2 \pi}} \frac{1}{z} \exp\left(-\frac{z^2}{2}\right), \quad z > 0.
\]
The result is obtained by replacing $z$ by $ t / \sigma_0$. We have that
\begin{align*}
 \Prob(Z \geq z) = \frac{1}{\sqrt{2 \pi}} \int_z^\infty \exp\left(-\frac{x^2}{2}\right) \, \d x \leq \frac{1}{\sqrt{2 \pi}} \int_z^\infty \frac{x}{z} \exp\left(-\frac{x^2}{2}\right) \, \d x = \frac{1}{\sqrt{2 \pi}} \frac{1}{z} \exp\left(-\frac{z^2}{2}\right),
\end{align*}
given that, for all $x \geq z > 0$, $1 \leq x / z$.
\end{proof}

\section{Simple proof in the context of GLMs}\label{sec:robustGLM}

In this section, we present an application of the proof technique of \autoref{sec:proof} in another context than linear regression, namely in the context of generalized linear models (GLMs). More precisely, using the same technique, we prove a robustness characterization result for a robust heavy-tailed version of gamma GLM. This robust model has been introduced and studied in \cite{gagnon2023GLM}. The motivation for using this model is the same as for using the robust linear regression models presented in this paper: gamma GLM is non-robust to outliers and thus a robust version is useful in situations where the data set at hand is suspected to contain outliers. In \cite{gagnon2023GLM}, a robustness characterization result is presented, but again the proof is highly technical and lengthy, due to the generality of the prior distribution. Here, we present a significantly simpler and intuitive proof by leveraging a specific prior distribution structure, as in \autoref{sec:proof}. The prior PDF allows, in particular, to exploit a tail bound.

\subsection{Model definition}\label{sec:GLMmodel}

As in the context of linear regression, we assume that we have access to a data set of the form $\{\bx_i, y_i\}$, where $\bx_1, \ldots, \bx_n \in \re^p$ are vectors of explanatory variable data points and $y_1, \ldots, y_n$ are observations of a dependent variable. In the case of gamma GLM, however, it is assumed that $y_1, \ldots, y_n > 0$. Also, it is assumed that $y_1, \ldots, y_n$ are realizations of $n$ random variables $Y_1, \ldots, Y_n$, where $Y_i / \mu_i \sim f_{\nu, c}$ with $\mu_i = \exp(\bx_i^T \bbeta)$ and $f_{\nu, c}$ is a heavy-tailed version of the gamma PDF:
\begin{align}\label{eq:proposed}
 f_{\nu, c}(z) = \begin{cases}
    f_{\text{mid}}(z) := \exp(-\nu z) z^{\nu - 1} \nu^\nu / \Gamma(\nu) \quad \text{if} \quad \zl \leq z \leq \zr, \cr
    f_{\text{right}}(z) := f_{\text{mid}}(\zr) \frac{\zr}{z} \left(\frac{\log \zr}{\log z}\right)^{\lambdar} \quad \text{if} \quad z > \zr, \cr
    f_{\text{left}}(z) := f_{\text{mid}}(\zl) \frac{\zl}{z} \left(\frac{\log \zl}{\log z}\right)^{\lambdal} = f_{\text{mid}}(\zl) \frac{\zl}{z} \left(\frac{\log(1 / \zl)}{\log(1 / z)}\right)^{\lambdal} \quad \text{if} \quad 0 < z < \zl,
 \end{cases}
\end{align}
$\zr, \lambdar, \zl$ and $\lambdal$ being functions of $\nu > 0$ and $c > 0$ given by
\begin{align*}
 & \zr := 1 + c / \sqrt{\nu}, \quad \zl := \begin{cases}
                                                                0 \quad \text{if} \quad \nu \leq 1, \cr
                                                                \max\{0, 1 - c / \sqrt{\nu}\} \quad \text{if} \quad \nu > 1,
                                                            \end{cases} \cr
 & \lambdar := 1 + \frac{f_{\text{mid}}(\zr) \log(\zr) \, \zr}{\Prob(Z_\nu > \zr)}, \quad \text{and} \quad \lambdal := 1 - \frac{f_{\text{mid}}(\zl) \log(\zl) \, \zl}{\Prob(Z_\nu < \zl)} = 1 + \frac{f_{\text{mid}}(\zl) \log(1 / \zl) \, \zl}{\Prob(Z_\nu < \zl)}.
\end{align*}
The random variable $Z_\nu$ follows a gamma distribution whose mean and shape parameters are 1 and $\nu$, respectively. For a detailed description of the model, see \cite{gagnon2023GLM}. Gamma GLM is retrieved by setting $\zl = 0$ and $\zl = +\infty$. The model is parametrized by using a mean parameter $\mu_i$ and a shape parameter $\nu$, both of which being considered unknown. In $f_{\nu, c}$, $c$ is a tuning parameter typically chosen by the user that allows to reach a compromise between efficiency and robustness. In \cite{gagnon2023GLM}, it is identified that $c = 1.6$ offers a good balance between efficiency and robustness.

\subsection{Robustness characterization result}\label{sec:resultGLM}

As in \autoref{sec:result}, to characterize the robustness of the model presented in \autoref{sec:GLMmodel} we consider an asymptotic regime where the outliers move further and further away from the bulk of the data along particular paths. In the context of gamma GLM, we consider that the outliers $(\bx_i, y_i)$ are such that $y_i \rightarrow \infty$ or $y_i \rightarrow 0$ with $\bx_i$ being kept fixed (but perhaps extreme). We refer to a couple $(\bx_i, y_i)$ with $y_i \rightarrow \infty$ as a \textit{large} outlier, and to a couple with $y_i \rightarrow 0$ as a \textit{small} outlier. The $y_i$ component is referred to as a \textit{large/small} outlying observation. We consider that each outlying observation goes to $\infty$ or 0 as its own specific rate. More specifically, for a large outlying observation, we consider that $y_i = b_i \omega$, and that $y_i = 1 / b_i \omega$ for a small outlying observation, with $b_i \geq 1$ a constant, and we let $\omega \rightarrow \infty$. For each non-outlying observation, we assume that $y_i = a_i$, where $a_i > 0$ is a constant.

As mentioned in \autoref{sec:result}, the limiting behaviour of the conditional PDF of $Y_i$ evaluated at an outlying point is central to the characterization of the robustness properties. We now present a proposition about this limiting behaviour in the case of robust gamma GLM.

\begin{Proposition}\label{prop:limit_PDFGLM}
For all $c > 0$, $\nu > 0$ and $\bbeta \in \re^p$,
\[
 \lim_{y_i \rightarrow \infty} \frac{ f_{\nu,c}(y_i/\mu_i)/\mu_i}{f_{\nu,c}(y_i)} = 1,
 \]
 recalling that $\mu_i = \exp(\bx_i^T \bbeta)$. If $\nu>1$ and $c<\sqrt{\nu}$ (the condition under which $f_{\text{left}}$ exists),
 \[
 \lim_{y_i \rightarrow 0} \frac{ f_{\nu,c}(y_i/\mu_i)/\mu_i}{f_{\nu,c}(y_i)} = 1.
 \]
\end{Proposition}

See \cite{gagnon2023GLM} for the proof of \autoref{prop:limit_PDFGLM}. There is an important difference between \autoref{prop:limit_PDFGLM} and \autoref{prop:limit_PDF}. The term $f_{\nu,c}(y_i)$ in the denominator in the limits above cannot be written as a product of two terms with one depending on $\nu$ but not on $y_i$ and the other one depending on $y_i$ but not on $\nu$. In \autoref{prop:limit_PDF}, the term in the denominator in the limit is $g(\sigma) f(y_i)$, thus being a product of two terms with one depending on $\sigma$ but not on $y_i$ and the other one depending on $y_i$ but not on $\sigma$. In other words, we are able to separate the parameters from the limiting object, which allows to proceed using our proof technique. To prove a robustness characterization result for the model in \autoref{sec:GLMmodel} (using this proof technique), we consider a simplifying situation as in \cite{gagnon2023GLM} where the parameter $\nu$ is considered known and fixed, like $c$; the only unknown parameter is thus $\bbeta$. The prior and posterior are thus about this parameter only, and they will be denoted by $\pi$ and $\pi_\omega(\, \cdot \mid \by)$, respectively. We further simplify by considering that $\nu$ is such that $\nu > 1$ and $c < \sqrt{\nu}$ to ensure the existence of both tails in our model, noting that $\nu > 1$ corresponds to the gamma PDF shape that often is sought for and supported by the data in applications (e.g., in actuarial science). The simplifying situation with $\nu$ fixed can be seen as an approximation of that where $\nu$ is considered unknown and random, but with a posterior mass that concentrates strongly around a specific value. The result that is obtained suggests that the posterior density when both $\bbeta$ and $\nu$ are considered unknown asymptotically behaves like one where the PDF terms of the outlying data points are each replaced by $f_{\nu,c}(y_i)$.

We now present our assumption on the prior distribution which will allow to proceed with a simple proof for the robustness characterization result.
\begin{Assumption}\label{ass:priorGLM}
   The components of $\bbeta = (\beta_1, \ldots, \beta_p)^T$ are independent and each $\beta_j$ has a sub-exponential distribution.
\end{Assumption}

In our Bayesian framework, it is assumed that the random variables $Y_1, \ldots, Y_n$ are conditionally independent given $\bbeta$. Therefore,
\[
 \pi_\omega(\bbeta \mid \by) = \pi(\bbeta) \prod_{i = 1}^n \frac{1}{\mu_i} f_{\nu, c}\left(\frac{y_i}{\mu_i}\right) \Bigg/ m(\by), \quad \bbeta \in \re^p,
\]
where
\[
 m_\omega(\by) = \int_{\re^p} \pi(\bbeta) \prod_{i = 1}^n \frac{1}{\mu_i} f_{\nu, c}\left(\frac{y_i}{\mu_i}\right)  \d\bbeta,
\]
if $m_\omega(\by) < \infty$.

We now present definitions that will allow to state the robustness characterization result. As in \autoref{sec:result}, let us define the index set of outlying data points by $\Oset$. The index set of non-outlying data points is thus given by: $\Oset^\mathsf{c} := \{1, \ldots, n\} \setminus \Oset$. We also define the set of non-outlying observations: $\by_{\Oset^\mathsf{c}} := \{y_i: i \in \Oset^\mathsf{c}\}$. A conclusion of our theoretical result is a convergence of the posterior distribution towards $\pi(\, \cdot \mid \by_{\Oset^\mathsf{c}})$, which has a density defined as follows:
\[
 \pi(\bbeta \mid \by_{\Oset^\mathsf{c}}) = \pi(\bbeta) \prod_{i \in \Oset^\mathsf{c}} \frac{1}{\mu_i} f_{\nu, c}\left(\frac{y_i}{\mu_i}\right) \Bigg/ m(\by_{\Oset^\mathsf{c}}), \quad \bbeta \in \re^p,
\]
where
\[
 m(\by_{\Oset^\mathsf{c}}) = \int_{\re^p} \pi(\bbeta) \prod_{i \in \Oset^\mathsf{c}} \frac{1}{\mu_i} f_{\nu, c}\left(\frac{y_i}{\mu_i}\right) \d\bbeta,
\]
if $m(\by_{\Oset^\mathsf{c}}) < \infty$.

As in \autoref{sec:result}, to obtain a convergence result, we need a guarantee that the aforementioned posterior distributions are proper. We now provide such a guarantee.

\begin{Proposition}\label{prop:properGLM}
   Suppose that \autoref{ass:priorGLM} holds. Then, $m(\mathbf{y}_{\Oset^\mathsf{c}}) < \infty$ and $m_\omega(\by) < \infty$ for all $\omega$.
\end{Proposition}

\begin{proof}
 We have that
 \begin{align*}
  m_\omega(\by) = \int_{\re^p} \pi(\bbeta) \prod_{i = 1}^n \frac{1}{\mu_i} f_{\nu, c}\left(\frac{y_i}{\mu_i}\right)  \d\bbeta \leq \prod_{i = 1}^n \frac{(\e^{-1}\nu)^\nu}{y_i \Gamma(\nu)} \int_{\re^p} \pi(\bbeta) \, \d\bbeta = \prod_{i = 1}^n \frac{(\e^{-1}\nu)^\nu}{y_i \Gamma(\nu)},
 \end{align*}
 which is finite for all $\omega$ (recalling that $y_i = b_i \omega$ for a large outlying observation, that $y_i = 1 / b_i \omega$ for a small outlying observation, and that $y_i = a_i$ for a non-outlying observation). In the inequality, we used \autoref{lemma:1GLM} (presented below in \autoref{sec:lemmas_GLM}) and, in the final equality, we used \autoref{ass:priorGLM}.

 The proof that $m(\mathbf{y}_{\Oset^\mathsf{c}}) < \infty$ is similar.
\end{proof}

We are now ready to present the robustness characterization result.

\begin{Theorem}\label{thm:robustnessGLM}
 Assume that $\nu$ is fixed and such that $\nu > 1$ and $c < \sqrt{\nu}$. Suppose that \autoref{ass:priorGLM} holds. As $\omega \rightarrow \infty$,
 \begin{enumerate}

  \item[(a)] the asymptotic behaviour of the marginal distribution is: $m_\omega(\by) / \prod_{i \in \Oset} f_{\nu,c}(y_i) \rightarrow m(\by_{\Oset^\mathsf{c}})$;

  \item[(b)] the posterior density converges pointwise: for any $\bbeta \in \re^p, \pi_\omega(\bbeta \mid \by) \rightarrow \pi(\bbeta \mid \by_{\Oset^\mathsf{c}})$;

  \item[(c)] the posterior distribution converges: $\pi_\omega(\, \cdot \mid \by) \rightarrow \pi(\, \cdot \mid \by_{\Oset^\mathsf{c}})$.
 \end{enumerate}
\end{Theorem}

\begin{proof}
 We proceed as in the proof of \autoref{Thm} and show how the key steps in it are adapted to prove a robustness characterization result for another model than linear regression. We start with the proof of Result (c) (assuming Result (b)). Next, we prove Result (b) (assuming Result (a)). Finally, we provide the proof of Result (a), which is longer.

  Result (c) is a direct consequence of Result (b) by Scheffé's lemma. To prove Result (b), we rewrite $\pi_\omega(\bbeta \mid \by)$ for fixed $\bbeta \in \re^p$ in order to exploit Result (a) and \autoref{prop:limit_PDFGLM}:
        \[
 \pi_\omega(\bbeta \mid \by) = \pi(\bbeta\mid\by_{\Oset^\mathsf{c}}) \, \frac{m(\by_{\Oset^\mathsf{c}}) \prod_{i \in \Oset} f_{\nu,c}(y_i)}{m_\omega(\by)} \prod_{i \in \Oset } \frac{f_{\nu,c}(y_i/\mu_i)/\mu_i}{f_{\nu,c}(y_i)}.
\]
 For any $\bbeta \in \re^p$,
\[
  \frac{m(\by_{\Oset^\mathsf{c}}) \prod_{i \in \Oset} f_{\nu,c}(y_i)}{m_\omega(\by)} \prod_{i \in \Oset } \frac{f_{\nu,c}(y_i/\mu_i)/\mu_i}{f_{\nu,c}(y_i)} \rightarrow 1,
\]
by Result (a) and \autoref{prop:limit_PDFGLM}.

We now prove Result (a) by showing that
\[
 \frac{m_\omega(\by)}{m(\by_{\Oset^\mathsf{c}})\prod_{i \in \Oset} f_{\nu,c}(y_i)} \rightarrow 1.
\]
We combine the numerator and the denominator in this expression to obtain an integral involving the same expression as in \autoref{prop:limit_PDFGLM}:
\begin{align*}
\frac{m_\omega(\by)}{m(\by_{\Oset^\mathsf{c}})\prod_{i \in \Oset} f_{\nu,c}(y_i)} &= \frac{m_\omega(\by)}{m(\by_{\Oset^\mathsf{c}})\prod_{i \in \Oset}  f_{\nu,c}(y_i)}
      \int_{\re^p} \pi_\omega(\bbeta \mid \by)  \, \d\bbeta \nonumber \\
   &=   \int_{\re^p}\frac{\pi(\bbeta) \prod_{i=1}^{n}
        f_{\nu,c}(y_i/\mu_i)/\mu_i}{m(\by_{\Oset^\mathsf{c}}) \prod_{i \in \Oset} f_{\nu,c}(y_i)} \, \d\bbeta \nonumber \\
   &=   \int_{\re^p}  \pi(\bbeta\mid \by_{\Oset^\mathsf{c}})  \prod_{i \in \Oset } \frac{f_{\nu,c}(y_i/\mu_i)/\mu_i}{f_{\nu,c}(y_i)} \, \d\bbeta =: I(\omega).
\end{align*}
By \autoref{prop:limit_PDFGLM}, we would obtain the result, that is $\lim_{\omega \rightarrow \infty} I(\omega) = 1$, if we were allowed to interchange the limit and the integral. As in the proof of \autoref{Thm}, we essentially prove that we are allowed to do so.

The form of $I(\omega)$ suggests the use of results like Lebesgue's dominated convergence theorem to prove Result (a). Analogously as in the proof of \autoref{Thm}, in the case where $y_i/\mu_i = \exp(\log(y_i) - \bx_i^T \bbeta)$ is of the order of $\omega$ for a large outlying observation (or of the order of $1 / \omega$ for a small outlying observation), we expect to be able to bound
\[
 \frac{f_{\nu,c}(y_i/\mu_i)/\mu_i}{f_{\nu,c}(y_i)}
\]
in a way that it does not depend on $\omega$ given the form of the tails of $f_{\nu,c}$ (see \autoref{sec:GLMmodel}); recall that $y_i$ is of the order of $\omega$ or $1/\omega$ for $i \in \Oset$. We follow this strategy and define a set for $\bbeta$ on which it is guaranteed that $y_i/\mu_i$ is of the order of $\omega$ for a large outlying observation and of the order of $1 / \omega$ for a small outlying observation:
\[
    S(\omega) := \bigcap_{i=1}^n \left\{\bbeta: |\bx_i^T\bbeta| \leq \log(\omega) / 2\right\}.
\]
Notice the similarity with the set with the same notation in \autoref{sec:proof}. The definition is motivated by the form of $y_i/\mu_i = \exp(\log(y_i) - \bx_i^T \bbeta)$.

We write
\[
    I(\omega) = I_1(\omega) + I_2(\omega),
\]
where
\[
    I_1(\omega) = \int_{\re^p} \ind_{S(\omega)} \,
        \pi(\bbeta\mid \by_{\Oset^\mathsf{c}}) \prod_{i \in \Oset } \frac{f_{\nu,c}(y_i/\mu_i)/\mu_i}{f_{\nu,c}(y_i)} \, \d\bbeta,
\]
and $I_2(\omega)$ is the integral on $S(\omega)^\mathsf{c}$. Note that $\ind_{S(\omega)} \rightarrow \ind_{\re^p}$ as $\omega \rightarrow \infty$ given that, for any $\bbeta \in \re^p$, there exists $\omega$ large enough so that $|\mathbf{x}_i^T\boldsymbol\beta| \leq \log(\omega) / 2$ for all $i$.

Similarly as in the proof of \autoref{Thm}, we now show that, on $S(\omega)$, the integrand in $I(\omega)$ is bounded by $\pi(\bbeta)$ times a constant, which does not depend on $\omega$ and is integrable (under \autoref{ass:priorGLM}). This implies that $\lim_{\omega \rightarrow \infty} I_1(\omega) = 1$ by Lebesgue’s dominated convergence theorem (and \autoref{prop:limit_PDFGLM}). Next, on $S(\omega)^\mathsf{c}$, we exploit the prior distribution structure to prove that $\lim_{\omega \rightarrow \infty} I_2(\omega) = 0$, which will allow to conclude that $\lim_{\omega \rightarrow \infty} I(\omega) = 1$.

 For $\bbeta \in S(\omega)$,
 \begin{align}
  \pi(\bbeta\mid \by_{\Oset^\mathsf{c}}) \prod_{i \in \Oset } \frac{f_{\nu,c}(y_i/\mu_i)/\mu_i}{f_{\nu,c}(y_i)} &\propto \pi(\bbeta) \prod_{i \in \Oset^\mathsf{c}} \frac{1}{\mu_i} f_{\nu, c}\left(\frac{y_i}{\mu_i}\right) \prod_{i \in \Oset } \frac{f_{\nu,c}(y_i/\mu_i)/\mu_i}{f_{\nu,c}(y_i)} \cr
  &\leq \pi(\bbeta) \prod_{i \in \Oset^\mathsf{c}}\frac{(\e^{-1}\nu)^\nu}{a_i \Gamma(\nu)} \prod_{i \in \Oset } \frac{f_{\nu,c}(y_i/\mu_i)/\mu_i}{f_{\nu,c}(y_i)} \cr
  &\leq \pi(\bbeta) \prod_{i \in \Oset^\mathsf{c}}\frac{(\e^{-1}\nu)^\nu}{a_i \Gamma(\nu)} \, 4^{|\Oset| \lambdal}, \label{eqn:boundGLM}
 \end{align}
using in the first line that $m(\mathbf{y}_{\Oset^\mathsf{c}}) < \infty$ (\autoref{prop:properGLM}), \autoref{lemma:1GLM} in the second line with $y_i = a_i$ for all $i \in \Oset^\mathsf{c}$, and finally that
\[
 \frac{f_{\nu,c}(y_i/\mu_i)/\mu_i}{f_{\nu,c}(y_i)} \leq 4^{\lambdal}
\]
for all $i \in  \Oset$, as we now explain.

On $\bbeta \in S(\omega)$, for $i \in \Oset$ with $y_i = b_i \omega$ a large outlying observation,
\begin{align*}
 \log(y_i / \mu_i) = \log(b_i) + \log(\omega) - \bx_i^T \bbeta \geq \log(\omega) - |\bx_i^T \bbeta| \geq \log(\omega) / 2,
\end{align*}
using that $b_i \geq 1$. Therefore, for $\omega$ large enough, we are guaranteed that $f_{\nu,c}(y_i/\mu_i)$ is evaluated on its right tail (see \autoref{sec:GLMmodel}), like $f_{\nu,c}(y_i)$:
\begin{align*}
 \frac{f_{\nu,c}(y_i/\mu_i)/\mu_i}{f_{\nu,c}(y_i)} = \frac{f_{\text{right}}(y_i/\mu_i)/\mu_i}{f_{\text{right}}(y_i)} &= \frac{f_{\text{mid}}(\zr) \frac{\zr}{y_i} \left(\frac{\log(\zr)}{\log(y_i / \mu_i)}\right)^{\lambdar}}{f_{\text{mid}}(\zr) \frac{\zr}{y_i} \left(\frac{\log \zr}{\log y_i}\right)^{\lambdar}} \cr
 &= \left(\frac{\log(y_i)}{\log(y_i / \mu_i)}\right)^{\lambdar} \cr
 &\leq \left(\frac{\log(b_i) + \log(\omega)}{\log(\omega)/2}\right)^{\lambdar} \cr
 &= \left(2\left(\frac{\log(b_i)}{\log(\omega)} + 1\right)\right)^{\lambdar} \leq 4^{\lambdal},
\end{align*}
using in the first two equalities the definition of $f_{\nu,c}$ (see \autoref{sec:GLMmodel}), in the first inequality that $\log$ is a strictly increasing function, and in the final inequality that $\log(b_i)/\log(\omega) \leq 1$ and $ \lambdar \leq \lambdal$ (see \cite{gagnon2023GLM}). We proceed similarly for the case where $i \in \Oset$ with $y_i = 1 / b_i \omega$ a small outlying observation:
\begin{align*}
 \log((y_i / \mu_i)^{-1}) = \log(b_i) + \log(\omega) + \bx_i^T \bbeta \geq \log(\omega) - |\bx_i^T \bbeta| \geq \log(\omega) / 2,
\end{align*}
using that $b_i \geq 1$, which implies that $y_i/\mu_i \leq 1 / \omega^{1/2}$. Therefore, for $\omega$ large enough, we are guaranteed in this case that $f_{\nu,c}(y_i/\mu_i)$ is evaluated on its left tail (see \autoref{sec:GLMmodel}), like $f_{\nu,c}(y_i)$:
\begin{align*}
 \frac{f_{\nu,c}(y_i/\mu_i)/\mu_i}{f_{\nu,c}(y_i)} = \frac{f_{\text{left}}(y_i/\mu_i)/\mu_i}{f_{\text{left}}(y_i)} &= \frac{f_{\text{mid}}(\zl) \frac{\zl}{y_i} \left(\frac{\log(1 / \zl)}{\log(\mu_i / y_i)}\right)^{\lambdal}}{f_{\text{mid}}(\zl) \frac{\zl}{y_i} \left(\frac{\log(1 / \zl)}{\log(1 / y_i)}\right)^{\lambdal}} \cr
 &\leq \left(\frac{\log(b_i) + \log(\omega)}{\log(\omega)/2}\right)^{\lambdal} \leq 4^{\lambdal},
\end{align*}
using the same arguments as for the case where $y_i = b_i \omega$ is a large outlying observation, except that we do not need to use $ \lambdar \leq \lambdal$.

Thus, using \eqref{eqn:boundGLM}, we have an upper bound on the integrand in $I_1(\omega)$ given by $\pi(\bbeta)$ times a constant, which is integrable under \autoref{ass:priorGLM}. Therefore, by Lebesgue's dominated convergence theorem and \autoref{prop:limit_PDFGLM}, $\lim_{\omega \rightarrow \infty} I_1(\omega) = 1$.

We now turn to proving that $\lim_{\omega \rightarrow \infty} I_2(\omega) = 0$. We have that
\begin{align*}
 &\int_{\re^p} \ind_{S(\omega)^\mathsf{c}} \,
        \pi(\bbeta\mid \by_{\Oset^\mathsf{c}}) \prod_{i \in \Oset } \frac{f_{\nu,c}(y_i/\mu_i)/\mu_i}{f_{\nu,c}(y_i)} \, \d\bbeta \cr
 &\quad \propto \int_{\re^p} \ind_{S(\omega)^\mathsf{c}} \,
        \pi(\bbeta) \prod_{i \in \Oset^\mathsf{c}} \frac{1}{\mu_i} f_{\nu, c}\left(\frac{y_i}{\mu_i}\right) \prod_{i \in \Oset } \frac{f_{\nu,c}(y_i/\mu_i)/\mu_i}{f_{\nu,c}(y_i)} \, \d\bbeta \cr
 &\quad \leq  \int_{\re^p} \ind_{S(\omega)^\mathsf{c}} \,
        \pi(\bbeta) \prod_{i \in \Oset^\mathsf{c}} \frac{(\e^{-1}\nu)^\nu}{a_i \Gamma(\nu)} \prod_{i \in \Oset } \frac{(\e^{-1}\nu)^\nu/(y_i\Gamma(\nu))}{f_{\nu,c}(y_i)} \, \d\bbeta \cr
 &\quad = \prod_{i \in \Oset^\mathsf{c}} \frac{(\e^{-1}\nu)^\nu}{a_i \Gamma(\nu)} \prod_{i \in \Oset } \frac{(\e^{-1}\nu)^\nu/(y_i\Gamma(\nu))}{f_{\nu,c}(y_i)} \, \Prob\left(\bigcup_{i=1}^n\left\{\bbeta: |\bx_i^T\bbeta| > \log(\omega) / 2\right\}\right),
\end{align*}
using \autoref{lemma:1GLM} in the inequality with $y_i = a_i$ for all $i \in \Oset^\mathsf{c}$.

We finish the proof by showing that $\Prob\left(\bigcup_{i=1}^n\left\{\bbeta: |\bx_i^T\bbeta| > \log(\omega) / 2\right\}\right)$ goes to 0 more quickly than
\[
 \prod_{i \in \Oset } \frac{(\e^{-1}\nu)^\nu/(y_i\Gamma(\nu))}{f_{\nu,c}(y_i)}
\]
 goes to infinity. Similarly as in the proof of \autoref{Thm} and as in \autoref{sec:alternativeprior},
\begin{align*}
\Prob\left(\bigcup_{i=1}^n \left\{\bbeta: |\bx_i^T\bbeta| > \log(\omega) / 2\right\} \right) &\leq \sum_{i=1}^n  \Prob\left\{\bbeta: |\bx_i^T\bbeta| > \log(\omega) / 2\right\} \cr
 &= \sum_{i=1}^n  \Prob\left\{\bbeta: |\bx_i^T\bbeta| > \omega_0 / 2\right\} \cr
 &\leq  \sum_{i=1}^n  2 \exp\left(- c_1 \, \frac{\omega_0 / 2 - \bx_i^T \E[\bbeta]}{K \|\bx_i\|_{\infty}}\right)
\end{align*}
using the union bound in the first line, the definition $\omega_0 = \log \omega$ in the second line and the tail bound in \autoref{lemma:2GLM} (presented below in \autoref{sec:lemmas_GLM}) in the final line, where $c_1 > 0$ is an absolute constant, $K = \max_j \|\beta_j - \E[\beta_j]\|_{\psi_1}$ and $\|\, \cdot \, \|_\infty$ is the infinity norm, $\|\, \cdot \,\|_{\psi_1}$ being the (finite) \textit{sub-exponential norm} \citep[Definition 2.7.5]{vershynin2018high}. \autoref{lemma:2GLM} is essentially an application of Theorem 2.8.2 in \cite{vershynin2018high} where the difference is that we account for the fact that $\beta_j$ does not necessarily have a mean of 0. Theorem 2.8.2 in \cite{vershynin2018high} can be seen as a statement that the distribution of a linear combination of mean-zero sub-exponential random variables has tails that behave like those of the distribution of one sub-exponential random variable.

Thus, $\Prob\left(\bigcup_{i=1}^n\left\{\bbeta: |\bx_i^T\bbeta| > \omega_0 / 2\right\}\right)$ goes to 0 exponentially quickly (in $\omega_0$). We now prove that
\[
 \prod_{i \in \Oset } \frac{(\e^{-1}\nu)^\nu/(y_i\Gamma(\nu))}{f_{\nu,c}(y_i)}
\]
goes to infinity polynomially quickly (in $\omega_0$), which will conclude the proof. For $y_i = b_i \omega$ a large outlying observation,
\begin{align*}
 \frac{(\e^{-1}\nu)^\nu/(y_i\Gamma(\nu))}{f_{\nu,c}(y_i)} = \frac{(\e^{-1}\nu)^\nu/(y_i\Gamma(\nu))}{f_{\text{right}}(y_i)} &=  \frac{(\e^{-1}\nu)^\nu/(y_i\Gamma(\nu))}{f_{\text{mid}}(\zr) \frac{\zr}{y_i} \left(\frac{\log \zr}{\log y_i}\right)^{\lambdar}} \cr
 &= \frac{(\e^{-1}\nu)^\nu}{\Gamma(\nu) f_{\text{mid}}(\zr) \, \zr} \left(\frac{\log(b_i) + \log(\omega)}{\log \zr}\right)^{\lambdar},
\end{align*}
using in the first two equalities the definition of $f_{\nu,c}$ (see \autoref{sec:GLMmodel}). With $\omega_0 = \log \omega$, we observe that the speed of the increase of this term is polynomial in $\omega_0$. We also have a polynomial increase in $\omega_0$ for $y_i = 1 / b_i \omega$ a small outlying observation:
\begin{align*}
 \frac{(\e^{-1}\nu)^\nu/(y_i\Gamma(\nu))}{f_{\nu,c}(y_i)} = \frac{(\e^{-1}\nu)^\nu/(y_i\Gamma(\nu))}{f_{\text{left}}(y_i)} &=  \frac{(\e^{-1}\nu)^\nu/(y_i\Gamma(\nu))}{f_{\text{mid}}(\zl) \frac{\zl}{y_i} \left(\frac{\log(1 / \zl)}{\log(1 / y_i)}\right)^{\lambdal}} \cr
 &= \frac{(\e^{-1}\nu)^\nu}{\Gamma(\nu) f_{\text{mid}}(\zl) \, \zl} \left(\frac{\log(b_i) + \log(\omega)}{\log(1 / \zl)}\right)^{\lambdal},
\end{align*}
using again in the first two equalities the definition of $f_{\nu,c}$ (see \autoref{sec:GLMmodel}). This concludes the proof.
\end{proof}

\subsection{Two lemmas}\label{sec:lemmas_GLM}

In this section, we present two lemmas used in the proof of \autoref{thm:robustnessGLM}.

\begin{Lemma}\label{lemma:1GLM}
Viewed as a function of $\mu > 0$, $f_{\nu,c}(y/\mu)/\mu$ is strictly increasing on $(0, y)$ and then strictly decreasing on $(y, \infty)$, for all $\nu, c, y > 0$. It is thus unimodal with a mode at $\mu = y$, and in particular, it is bounded above by $(\e^{-1}\nu)^\nu/(y\Gamma(\nu))$.
\end{Lemma}

See \cite{gagnon2023GLM} for the proof of \autoref{lemma:1GLM}.

\begin{Lemma}\label{lemma:2GLM} 
Assume that $\bbeta = (\beta_1, \ldots, \beta_p)^T \in \re^p$ is a random vector such that its components are independent and each $\beta_j$ has a sub-exponential distribution. For any fixed $\bx_i \in \re^p$ and large enough $\omega > 0$,
\[
  \Prob\left\{\bbeta: |\bx_i^T\bbeta| > \omega / 2\right\} \leq 2 \exp\left(- c_1 \, \frac{\omega / 2 - |\bx_i^T \E[\bbeta]|}{K \|\bx_i\|_{\infty}}\right),
\]
where $c_1 > 0$ is an absolute constant, $K = \max_j \|\beta_j - \E[\beta_j]\|_{\psi_1}$ and $\|\, \cdot \,\|_\infty$ is the infinity norm, $\|\, \cdot \,\|_{\psi_1}$ being the (finite) \textit{sub-exponential norm} \citep[Definition 2.7.5]{vershynin2018high}.
\end{Lemma}

\begin{proof}
  Let us consider that $\E[\bx_i^T \bbeta] = \bx_i^T \E[\bbeta] \geq 0$. We proceed symmetrically if $\bx_i^T \E[\bbeta] < 0$. For $\omega$ large enough,
 \begin{align*}
  \Prob\left\{\bbeta: |\bx_i^T\bbeta| > \omega / 2\right\} &= \Prob\left\{\bbeta: \bx_i^T\bbeta > \omega / 2 \text{ or } \bx_i^T\bbeta < -\omega / 2 \right\} \cr
  &= \Prob\left\{\bbeta: \bx_i^T\bbeta - \bx_i^T \E[\bbeta] > \omega / 2 - \bx_i^T \E[\bbeta] \text{ or } \bx_i^T\bbeta - \bx_i^T \E[\bbeta] < -\omega / 2 - \bx_i^T \E[\bbeta] \right\} \cr
  &\leq \Prob\left\{\bbeta: \bx_i^T\bbeta - \bx_i^T \E[\bbeta] > \omega / 2 - \bx_i^T \E[\bbeta] \text{ or } \bx_i^T\bbeta - \bx_i^T \E[\bbeta] < -\omega / 2 + \bx_i^T \E[\bbeta] \right\} \cr
  &= \Prob\left\{\bbeta: |\bx_i^T\bbeta - \bx_i^T \E[\bbeta]| > \omega / 2 - \bx_i^T \E[\bbeta]\right\} \cr
  &\leq 2 \exp\left(- c_1 \, \min\left\{\frac{(\omega / 2 - \bx_i^T \E[\bbeta])^2}{K^2 \|\bx_i\|^2}, \frac{\omega / 2 - \bx_i^T \E[\bbeta]}{K \|\bx_i\|_{\infty}}\right\}\right),
 \end{align*}
 using in the first inequality that $-\omega / 2 - \bx_i^T \E[\bbeta] \leq -\omega / 2 + \bx_i^T \E[\bbeta]$ and Theorem 2.8.2 of \cite{vershynin2018high} in the second inequality. For $\omega$ large enough,
 \[
  2 \exp\left(- c_1 \, \min\left\{\frac{(\omega / 2 - \bx_i^T \E[\bbeta])^2}{K^2 \|\bx_i\|^2}, \frac{\omega / 2 - \bx_i^T \E[\bbeta]}{K \|\bx_i\|_{\infty}}\right\}\right) = 2 \exp\left(- c_1 \, \frac{\omega / 2 - \bx_i^T \E[\bbeta]}{K \|\bx_i\|_{\infty}}\right).
 \]
\end{proof}

\section{Alternative to \autoref{ass:prior}}\label{sec:alternativeprior}

In this section, we show that \autoref{Thm} holds for an important class of prior distributions, with essentially the same proof.

\begin{Assumption}[Alternative to \autoref{ass:prior}]\label{ass:prioralternative}
   The prior distribution is such that $\bbeta$ and $\sigma$ are independent. The distribution of $\sigma^2$ has finite inverse moments. The components of $\bbeta = (\beta_1, \ldots, \beta_p)^T$ are independent and each $\beta_j$ has a sub-exponential distribution.
\end{Assumption}

 It can be readily verified that, up to the point where we prove that $\lim_{\omega \rightarrow \infty} I_2(\omega) = 0$, we can proceed as in the proof of \autoref{Thm} (\autoref{sec:proof}) because it is only required that the prior distribution of $\bbeta$ is proper and that the prior distribution of $\sigma^2$ has finite inverse moments, which holds under \autoref{ass:prioralternative}. When we prove that $\lim_{\omega \rightarrow \infty} I_2(\omega) = 0$, we can use the same arguments as in the proof of \autoref{Thm} to obtain
 \begin{align*}
  I_2(\omega) &= \int_{\re^p} \int_0^\infty  \ind_{S(\omega)^\mathsf{c}} \, \pi(\bbeta, \sigma\mid \by_{\Oset^\mathsf{c}})  \prod_{i \in \Oset } \frac{(1/\sigma)f((y_i - \bx_i^T \bbeta)/\sigma)}{g(\sigma) f(y_i)} \, \d \sigma \, \d\bbeta \cr
& \leq C^n \int_{\re^p} \int_0^\infty  \ind_{S(\omega)^\mathsf{c}} \, \pi(\bbeta, \sigma) \, \frac{1}{\sigma^n} \prod_{i \in \Oset } \frac{1}{f(2b_i\omega)} \, \d \sigma \, \d\bbeta.
 \end{align*}
 A difference with the proof of \autoref{Thm} is that, under \autoref{ass:prioralternative}, $\bbeta$ and $\sigma$ are independent, and therefore
 \begin{align*}
  &C^n \int_{\re^p} \int_0^\infty  \ind_{S(\omega)^\mathsf{c}} \, \pi(\bbeta, \sigma) \, \frac{1}{\sigma^n} \prod_{i \in \Oset } \frac{1}{f(2b_i\omega)} \, \d \sigma \d\bbeta \cr
  &\quad\propto \left(\prod_{i \in \Oset } \frac{1}{f(2b_i\omega)}\right) \E[\sigma^{-n}]  \Prob\left(\bigcup_{i=1}^n \left\{\bbeta: |\bx_i^T\bbeta| > \omega / 2\right\}\right).
 \end{align*}
 We have that $\E[\sigma^{-n}]$ is finite because $\sigma^2$ has finite inverse moments under \autoref{ass:prioralternative}. As mentioned in the proof of \autoref{Thm}, $\prod_{i \in \Oset } f(2b_i\omega)^{-1}$ goes to infinity polynomially quickly under \autoref{ass:PDF}. Therefore, we can conclude that $\lim_{\omega \rightarrow \infty} I_2(\omega) = 0$ if
 \[
  \Prob\left(\bigcup_{i=1}^n \left\{\bbeta: |\bx_i^T\bbeta| > \omega / 2\right\}\right)
 \]
 goes to 0 exponentially quickly, which we prove under \autoref{ass:prioralternative}. We have that
  \begin{align*}
  \Prob\left(\bigcup_{i=1}^n \left\{\bbeta: |\bx_i^T\bbeta| > \omega / 2\right\}\right) &\leq \sum_{i=1}^n \Prob\left\{\bbeta: |\bx_i^T\bbeta| > \omega / 2\right\} \cr
  &\leq 2 \exp\left(- c_1 \, \frac{\omega / 2 - \bx_i^T \E[\bbeta]}{K \|\bx_i\|_{\infty}}\right),
 \end{align*}
 using the union bound in the first inequality and the tail bound in \autoref{lemma:2GLM} in the second inequality, where $c_1 > 0$ is an absolute constant, $K = \max_j \|\beta_j - \E[\beta_j]\|_{\psi_1}$ and $\|\, \cdot \,\|_\infty$ is the infinity norm, $\|\, \cdot \,\|_{\psi_1}$ being the (finite) \textit{sub-exponential norm} \citep[Definition 2.7.5]{vershynin2018high}. As mentioned in \autoref{sec:resultGLM}, \autoref{lemma:2GLM} is essentially an application of Theorem 2.8.2 in \cite{vershynin2018high} where the difference is that we account for the fact that $\beta_j$ does not necessarily have a mean of 0. Theorem 2.8.2 in \cite{vershynin2018high} can be seen as a statement that the distribution of a linear combination of mean-zero sub-exponential random variables has tails that behave like those of the distribution of one sub-exponential random variable. This concludes the demonstration that \autoref{Thm} holds if we replace \autoref{ass:prior} by \autoref{ass:prioralternative}, while leaving the proof of \autoref{Thm} essentially unchanged.

\end{document}